\documentclass[10pt,oneside]{article}    

\usepackage{mathrsfs}
\usepackage{amsmath,amsfonts,amssymb,amsthm}
\usepackage{caption}
\usepackage{subfigure}
\usepackage{cite}
\usepackage{color}
\usepackage{graphicx}
\usepackage{subfigure}
\usepackage{float}
\usepackage{paralist}
\usepackage[colorlinks,
linkcolor=red,
citecolor=blue,
anchorcolor=red
]{hyperref}

\usepackage[T1]{fontenc}

\newtheorem{thm}{Theorem}[section]
\newtheorem{lem}[thm]{Lemma}
\newtheorem{prop}[thm]{Proposition}
\newtheorem{rmk}{Remark}[section]
\def\f{\frac}
\def\l{\lambda}
\def\a{\alpha}
\def\gm{\gamma}
\def\r{\rho }

\def\b{\bar}

\def\e{\eta}
\def\h{\phi}
\def\s{\psi}
\def\b{\bar}
\def\v{\varepsilon}
\def\wt{\widetilde}
\def\sg{\sigma}

\def\be{\begin{equation}}
\def\ee{\end{equation}}
\def\bs{\begin{split}}
	\def\es{\end{split}}	
\def\mb{\mathbb}
\def\mc{\mathcal}
\def\bma#1\ema{{\allowdisplaybreaks\begin{split}#1\end{split}}}

\numberwithin{equation}{section}

\begin{document}
\title{ {\LARGE \textbf{ Existence and Nonlinear Stability of  Steady-States to Outflow Problem for the  Full Two-Phase Flow 
 }}}

\footnotetext{*Corresponding author.}
\footnotetext{\emph{E-mail address:}\quad hailiang.li.math@gmail.com (H.-L. Li),\quad shuangzhaomath@163.com (S. Zhao),}
\footnotetext{zhw1880118@163.com (H.-W. Zuo).}
\author{{Hai-Liang  Li $^{a,b}$,\quad Shuang Zhao$^{a,b*}$,\quad Han-Wen Zuo$^{a,b}$}\\[2mm]
{ \it\small $^{a}$ School of Mathematical Sciences,
	Capital Normal University, Beijing 100048, P.R. China.}\\
{ \it\small $^{b}$ Academy for Multidisciplinary Studies, Capital Normal University, Beijing 100048, P.R. China.}}\date{}

\maketitle
\begin{abstract}
	The outflow  problem for the viscous full two-phase flow  model in a half line is investigated in the present paper.
	The existence, uniqueness and   nonlinear  stability of the steady-state are shown respectively  corresponding to the  supersonic, sonic or subsonic state at  far field. 
	This is different from the outflow problem for the isentropic Navier-Stokes equations, where
	there is no steady-state for the subsonic state.
	 Furthermore, we  obtain  either  exponential  time decay rates for the supersonic state  or  algebraic time decay rates for  supersonic  and sonic states  in weighted Sobolev spaces.	
 
\end{abstract}
\section{Introduction} 
 Two-phase flow models play  important roles in applied scientific areas, for instance,  nuclear, engines, chemical engineering, medicine, oil-gas,  fluidization,   waste water treatment, biomedical, liquid crystals,   lubrication \cite{IM,MV,HL,BPPS,LLP}, etc. 
In this paper, we consider the full two-phase flow model  which  
can be  formally obtained from a Vlasov-Fokker-Planck equation coupled with the compressible Navier-Stokes equations through the  Chapman-Enskog expansion \cite{LWW}.

We consider the initial-boundary value problem (IBVP) for the  full two-phase flow model as follows:
\begin{equation}
\left\{
\begin{split}
&                  
\rho_{t}+(\rho u)_{x}=0,\\&
(\rho u)_{t} +[\rho u^{2} + p_{1}(\rho)  ]_{x}      =\mu u_{xx}  +n(v-u),\\&
n_{t}+(n v)_{x}=0,\\&
(n v)_{t} +[n v^{2} + p_{2}(n)  ]_{x}      =(n v_{x})_{x}  -n(v-u), ~~~x>0,~t>0,
\label{f}  
\end{split}   
\right . 
\end{equation}  
where $\rho>0$ and $n>0$ stand for the densities,   $u$ and  $v$ are the velocities of two fluids respectively, the constant $\mu>0$ is the viscosity coefficient, and the pressure-density functions take  forms
\be 
p_{1}(\rho)=A_{1}\rho^{\gamma},\quad p_{2}(n)=A_{2}n^{\alpha} 
\ee
with   $A_{1}>0 ,~ A_{2}>0 $, $ \gamma \geq  1$ and $\alpha\geq 1$.   
 The initial data are given by
 \begin{equation}
 ( \rho,  u, n, v)(0,  x)=  ( \rho_{0}, u_{0}, n_{0}, v_{0})(x),      ~~~~\underset{x\in \mathbb{R}_{+}}{\inf}   \rho _{0}(x) > 0,    ~~~~\underset{x\in \mathbb{R}_{+}}{\inf}   n _{0}(x)  > 0 ,
  \label{initial c}
 \end{equation}
\begin{equation} \begin{split} & \lim_{x\to+\infty} ( \rho_{0}, u_{0}, n_{0}, v_{0})(x)        =(\rho_{+},u_{+}, n_{+}, u_{+}),     \quad \r_{+}>0,\quad n_{+}>0,
 \end{split} \end{equation}
  and the outflow  boundary condition is imposed by
 \begin{equation}
 (u,v)(t,0)=(u_{-},u_{-}),
\quad  u_{-}<0,
 \label{outflow c}
 \end{equation}
where $\rho_{+}>0$, $n_{+}>0$, $u_{+}$  and $u_{-}<0$ are constants.

 The   condition $u_{-}<0$ means that  the fluids flow out the region $\mb{R}_{+}$ through the boundary $x=0$ with the velocity $u_{-}$,  and therefore the problem $({\ref{f}})$-$({\ref{outflow c}})$ is called the outflow problem\cite{MA}. On the other hand, the similar problem with  the case  $u_{-}>0$ is called the inflow problem\cite{MA}, and the densities on the boundary $(\r,n)(t,0)=(\r_{-}, n_{-})$  are also imposed for the well-posedness of the inflow problem.

It is an interesting issue to study the outflow/inflow problem. There are many  important progress made recently  about   the existence and nonlinear stability of steady-states and basic waves to the outflow/inflow problem for one-phase flow, such as compressible Navier-Stokes equations \cite{HQ out,KNZ,NN,KNNZ out,NNU,NNY,CHS,HMS in,MN in,WQ in,HW in,NN in, KZ,HLS in, WQ in2, WWZ,WWZ1, FLWZ, KZ 1,Q}.
%
For instance, for the inflow problem of isentropic Navier-Stokes equations, Matsumura-Nishihara \cite{MN in}  showed the existence and stability of steady-states for both  subsonic and sonic cases together with the stability of the superposition of steady-states and rarefaction waves under small perturbation, while Fan-Liu-Wang-Zhao \cite{FLWZ} investigated that steady-states  and rarefaction waves  were nonlinear stable under large perturbation, and  Huang-Matsumura-Shi \cite{HMS in} obtained the stability of superposition of steady-states and shock waves under small perturbation.
For the inflow problem of full compressible Navier-Stokes equations, Nakamura-Nishibata\cite{NN in} got the existence and stability of steady-states also
for  both subsonic and sonic cases, 
and Qin-Wang \cite{WQ in2, WQ in} and  Hong-Wang \cite{HW in} showed the combination of steady-states and rarefaction waves. 
%
For the outflow problem of isentropic Navier-Stokes equations,  the  existence and  nonlinear stability  of steady-states  for both  supersonic and sonic cases under small perturbation were proved in \cite{KNZ,NN, NNU,NNY}, the convergence rates toward steady-states for supersonic and sonic cases were investigated in \cite{ NNU,NNY}, the stability of the superposition of steady-states and rarefaction waves 
were got in \cite{KZ, HQ out}.  
%
For the outflow problem of full compressible Navier-Stokes equations, Kawashima-Nakamura-Nishibata\cite{KNNZ out} established the existence and nonlinear stability of steady-states under small perturbation for three cases: supersonic, sonic and subsonic flow  
while there is no steady-state to the outflow problem of isentropic Navier-Stokes equations for the subsonic case, 
 in addition they also gained the convergence rates toward the stationary solutions for supersonic and sonic cases \cite{KNNZ out}, and Qin\cite{Q}, Wan-Wang-Zhao\cite{WWZ} and Wan-Wang-Zou\cite{WWZ1}   obtained the stability of steady-states, rarefaction waves and their combination under large perturbation.

  It is  of  interest and challenges 
   to investigate the outflow/inflow problem for two-phase flow models 
  due to the  coupled motions of two phases.
Although it is rather complicated, some important progress has been made about the existence and nonlinear stability of steady-states and basic waves to the  outflow/inflow problem for two-phase flow models  \cite{DY,YZZ out, ZL out,HSW in,YZ1,LZ}. For instance,  Yin-Zhu\cite{YZ1} obtained  the existence of steady-states 
 similar to isentropic Navier-Stokes equations\cite{KNZ},  and the  nonlinear stability  and convergence rates of steady-states for the  supersonic case  to the outflow problem of the drift-flux model.
  For the outflow problem of the  two-fluid Navier-Stokes-Poisson system,  the existence of stedy-states  
  which is similar to that of isentropic Navier-Stokes equations\cite{KNZ}, 
  and the nonlinear stability of rarefactions waves and steady-states together with the superposition of steady-states and rarefaction  waves    were proved in  \cite {DY,YZZ out}.  
We established existence and nonlinear stability of steady-states to inflow problem of the model $(\ref{f})$ for  supersonic, sonic and subsonic cases in \cite{LZ}, which is a 
different phenomena compared with the inflow problem for isentropic Navier-Stokes equations \cite{MN in} and full compressible Navier-Stokes equations\cite{NN in}, where there is no steady-state for the supersonic case.

However, there is no result about 
the existence and nonlinear  stability  of steady-states for the outflow problem $(\ref{f})$-$(\ref{outflow c})$. The main purpose of this paper is to prove the existence and  nonlinear  stability of steady-states for the  supersonic, sonic and subsonic case,  and obtain either  exponential  time decay rates for the supersonic flow  or  algebraic time decay rates for both supersonic  and sonic flows.  Contrary to the  isentropic Navier-Stokes equations, the steady-state to the IBVP $(\ref{f1})$-$(\ref{outflow c})$ exists  for the subsonic case.
 
The steady-state
$(\widetilde{\rho },\widetilde{u},\widetilde{n},\widetilde{v})(x)$  to the outflow  problem $(\ref{f})$-$(\ref{outflow c})$ 
satisfies the following system 
\begin{equation}
\left\lbrace 
\begin{split}
&
(  \widetilde{\rho  }   \widetilde{u}   )_{x}  =0,\\&
[\widetilde{\rho}  \widetilde{u}^{2}+p_{1}(\widetilde{\rho})]_{x}
=(\mu \widetilde{u}_{x})_{x}   +\widetilde{n}( \widetilde{v}- \widetilde{u}),\\&
(  \widetilde{n  }   \widetilde{v}   )_{x}  =0,\\&
[\widetilde{n}  \widetilde{v}^{2}+p_{2}(\widetilde{n})]_{x}
=(\widetilde{n}\widetilde{v}_{x})_{x}   -\widetilde{n}( \widetilde{v}- \widetilde{u}),
 ~~~x>0,~~~~~~~~~~~~~~~~~~~~~~~
\end{split}
\right. 
\label{stationary f}
\end{equation}
with  the boundary conditions and spatial far field conditions
\begin{equation}
\begin{split} 
&
(\widetilde{u},\widetilde{v})(0) =(u_{-},u_{-}),
\quad 
\lim_{x \to \infty}
(\widetilde{\rho },\widetilde{u},\widetilde{n },\widetilde{v })(x)
=(\rho_{+},u_{+},n_{+},u_{+}),
\quad
\underset{x\in \mathbb{R}_{+}}  {\inf }\widetilde{\r}(x)   >0,
\quad \underset{x\in \mathbb{R}_{+}}  {\inf }\widetilde{n}(x)   >0.
\label{stationary boundary c}
\end{split} 
\end{equation}
 Integrating $ (\ref{stationary f})_{1}$ and $ (\ref{stationary f})_{3}$  over $(x, +\infty)$, we have  
 \begin{equation}
 \begin{split}
 &
\widetilde{u}=\frac{\rho_{+}}{\widetilde{\rho }}u_{+}
,
\quad 
\widetilde{v}=\frac{n_{+}}{\widetilde{n }}u_{+},
\label{widetilde{n rho u}}
\end{split} 
 \end{equation}
which implies that the following relationship
 \begin{equation}
 u_{+}=\frac{\widetilde{n}(0)}{n_{+}} u_{-}=\frac{\widetilde{\rho}(0)}{\rho_{+}} u_{-}<0
 \label{u_{+}}
 \end{equation}
 is the  necessary  property of the steady-state to  the boundary value problem (BVP)
 $ (\ref{stationary f})$-$ (\ref{stationary boundary c}) $.
   
Define the  Mach number $M_{+}$ and the  sound speed $c_{+}$ as below
\be 
M_{+}:=\f{|u_{+}|}{c_{+}},\quad c_{+}:=(\f{A_{1}\gm\r_{+}^{\gm}+A_{2}\a n_{+}^{\a}}{\r_{+}+n_{+}})^{\f{1}{2}}.\label{c}
\ee 
Then, we have the  following results about the existence and uniqueness  of the  steady-state.
\begin{thm}
	\label{thm stationary s}
Let  $\delta:=|u_{-}-u_{+}|>0$  and $u_{+}<0$ hold. Then there exists  a set $\Omega_{-}\subset \mb{R}_{-}$  such that if $u_{-} \in \Omega_{-}$ and $\delta$ sufficiently small,   there exists a unique strong solution $(\widetilde{\rho },\widetilde{u},\widetilde{n},\widetilde{v})(x)$ to the BVP $(\ref{stationary f})$-$(\ref{stationary boundary c})$  which satisfies either for  the supersonic or subsonic case $M_{+}\neq 1$  that
 \begin{equation}
 ~~~~~~|   \partial_{x}^{k} (\widetilde{\rho}-\rho_{+},\widetilde{u}-u_{+},\widetilde{n}-n_{+},\widetilde{v}-u_{+})  |   \leq C_{1} \delta e^{-c_{0} x} ,~~
 ~k=0,1,2,3,
 \label{M_{+}>1 stationary solution d}
 \end{equation}
 or for the sonic case $M_{+}=1$  that    
 \begin{equation}
 |\partial_{x}^{k}  (\widetilde{\rho}-\rho_{+},\widetilde{u}-u_{+},\widetilde{n}-n_{+},\widetilde{v}-u_{+}) |  \leq  C_{2} \frac{\delta ^{k+1}}{(1+\delta x)^{k+1}},~~~k=0,1,2, 3,
 \label{sigma}
 \end{equation}
 and 
 \begin{equation} 
  (\widetilde{u}_{x},\widetilde{v}_{x})=(a\sigma^{2}(x),a \sigma^{2}(x))+O(|\sigma(x)|^{3}),
 \label{sigma 1}
 \end{equation}
where $\sigma(x)$ is a  smooth function satisfying $\sg_{x}=-a\sg^{2}+O(|\sg|^{3})$ and  
\begin{equation}
c_{1}\frac{\delta}{1+\delta x}\leq \sigma(x)\leq C_{3}\frac{\delta}{1+\delta x},~~~~
|\partial^{k}_{x}\sigma(x)|\leq  C_{3}\frac{\delta^{k+1}}{(1+\delta x)^{k+1}},~~~k=0,1,2,3,
\label{sig}
\end{equation}
and  $C_{i}>0$, $i=1,2,3$, $c_{0}>0$, $c_{1}>0$, and $a>0$  are positive constants.
\end{thm} 
\begin{rmk}
Due to the drag force term,  the existence of  steady-states to the IBVP $(\ref{f})$-$(\ref{outflow c})$  is  obtained even for the  subsonic case $M_{+}<1$, which is different  from that  of   the isentropic  Navier-Stokes equations {\rm \cite{KNZ}}.  Moreover,   the existence of  steady-states to the IBVP $(\ref{f})$-$(\ref{outflow c})$ is similar to  that  of the full  compressible Navier-Stokes equations  {\rm\cite{KNNZ out}}.
\end{rmk}
Then, we  have the  nonlinear stability of the steady-state to the IBVP $(\ref{f})$-$(\ref{outflow c})$  for supersonic, sonic   and subsonic cases.  
\begin{thm}\label{thm long time behavior} Let the same conditions in Theorem \ref{thm stationary s}  hold and assume that it holds 
\be |p'_{1}({\r_{+}})-p'_{2}({n_{+}})|	\leq  \sqrt{2}|u_{+}|\min\{ (1+\f{\r_{+}}{n_{+}})[(\gm-1)p'_{1}(\r_{+})]^{\f{1}{2}},(1+\f{n_{+}}{\r_{+}})[(\a-1)p'_{2}(n_{+})]^{\f{1}{2}} \}
\label{p small}
\ee 
 for the sonic case $M_{+}=1.$
	Then, there exists a small positive constant $\v_{0}>0$  such that if 
	\be
	\| (\rho_{0}-\widetilde{\rho}, u_{0}-\widetilde{u}, n_{0}-\widetilde{n},v_{0}-\widetilde{ v}) \|_{H^{1}}+\delta\leq \varepsilon_{0},
	\ee
	the IBVP $(\ref{f})$-$(\ref{outflow c})$ has a unique global solution $(\rho,u, n, v)(t,x)$ satisfying
 \begin{equation} 
 \left\lbrace  \begin{split}&
(\rho-\widetilde{\rho},u-\widetilde{u},n-\widetilde{n},v-\widetilde{v}   )\in C([0,+\infty);H^{1}),
\\
&
(\rho-\widetilde{\rho},n-\widetilde{n} )_{x}\in L^{2}([0,+\infty);L^{2}) ,\\& (u-\widetilde{u},v-\widetilde{v} )_{x} \in L^{2}([0,+\infty);H^{1}),
 \nonumber
 \end{split} 
 \right.  
 \end{equation} 
 and 
 \begin{equation}
 \lim_{t \to +\infty } \sup_{x\in \mathbb{R_{+}}}|(\rho-\wt{\r},u-\wt{u},n-\wt{n},v-\wt{v})(t,x)|=0.
 \end{equation}
\end{thm}

In addition, we have the time  convergence rates  of the  global solution to the IBVP  $(\ref{f})$-$(\ref{outflow c})$ for both supersonic and sonic cases.
 
\begin{thm}
	\label{thm time decay}
	
Assume that the same conditions in Theorem \ref{thm stationary s} hold. Then, the following results hold.
 \begin{itemize}
 \item[(i)]	For $M_{+}>1$ and $\l>0$, if the initial data satisfy
$$ (1+ x)^{\frac{\lambda}{2}}(\rho_{0}-\widetilde{\rho}, u_{0}-\wt{u}, n_{0}-\wt{u}, n-\wt{n} ) \in L^{2}(\mathbb{R_{+}})$$ 
and 
\be \| (\rho_{0}-\widetilde{\rho}, u_{0}-\widetilde{u}, n_{0}-\widetilde{n},v_{0}-\widetilde{ v}) \|_{H^{1}}+\delta\leq \varepsilon_{0}, \ee
for a small positive constant $\v_{0}>0$,  then the solution $(\rho, u, n, v)(t, x)$ to the IBVP $(\ref{f})$-$(\ref{outflow c})$ satisfies
\begin{equation}
\begin{aligned}
~~~~\|(\r-\wt{\r},u-\wt{u},n-\wt{n},v-\wt{v}) (t)\|_{L^{\infty}}\leq C_{4} \delta_{0} (1+ t)^{-\frac{\lambda}{2}},
\label{M_{+}>1 algebra d}
\end{aligned}
\end{equation}
where $C_{4}>0$  
and 
$ 
\delta_{0}:=\| (\rho_{0}-\widetilde{\rho}, u_{0}-\widetilde{u},n_{0}-\widetilde{n},v_{0}-\widetilde{v} )  \|_{H^{1}}
+ \|  (1+x)^{\frac{\lambda}{2}}(\rho_{0}-\widetilde{\rho}, u_{0}-\widetilde{u},n_{0}-\widetilde{n},v_{0}-\widetilde{v} )  \|_{L^{2}}$ are constants independent of time.
\item[(ii)] For $M_{+}=1$,  $1 \leq \lambda<\l^{*}, \l^{*}:=2+\sqrt{8+\f{1}{1+b^{2}}}$ and $b:=\f{\r_{+}(u_{+}^{2}-p'_{1}(\r_{+}))}{|u_{+}|\sqrt{(\mu+n_{+})n_{+}}}$,  
if  for arbitrary  $\nu \in (0,\l ]$, there  exists a small positive constant 
 $\varepsilon_{0}>0$
such that 
\be\|\sigma^{-\frac{\lambda }{2} }(\rho_{0} -\widetilde{\rho }, u_{0}-\widetilde{u}, n_{0}-\widetilde{n}, v_{0}-\widetilde{v})   \|_{H^{1}}+\delta^{\frac{1}{2}} \leq \varepsilon_{0},\ee
then the IBVP $(\ref{f})$-$(\ref{outflow c})$ has a unique global solution $(\rho,u, n, v)(t,x)$ satisfying
\be 
\left\lbrace 
\begin{aligned}
&	\sg^{-\f{\nu}{2}}(\r-\wt{\r}, u-\wt{u}, n-\wt{n}, v-\wt{v}) \in C([0,+\infty): H^{1}),
\\
&
\sg^{-\f{\nu}{2}} (\r-\wt{\r}, n-\wt{n})_{x} \in L^{2}([0, +\infty); L^{2}), 
\\
&
\sg^{-\f{\nu}{2}} (u-\wt{u}, v-\wt{v} )_{x} \in L^{2}([0,+\infty); H^{1} ),
\end{aligned}
\right. 
\ee 
and
\begin{equation}\| 
\sg^{-\f{\nu}{2}}(\rho-\widetilde{\rho}, u-\widetilde{u}, n-\widetilde{n}, v-\widetilde{v} )(t) \|_{H^{1}}
\leq 
C_{5} \delta_{1}
(1+t)^{-\frac{\lambda-\nu}{4}},
\label{M_{+}=1 algebra d}
\end{equation}
where $\sg(x)$ is defined by $(\ref{sig})$ satisfying $(\ref{sg0})$, $C_{5}>0$ is a positive constant independent of time, and $\delta_{1}:=\|  \sigma^{-\frac{\lambda}{2} }(\rho_{0}-\widetilde{\rho}, u_{0}-\widetilde{u}, n_{0}-\widetilde{n}, v_{0}-\widetilde{v}) \|_{H^{1}} $ is a  constant.
\end{itemize}
\end{thm}
%
\begin{rmk}
  For $M_{+}>1$,  the exponential time convergence rates of the global solution  to the IBVP $(\ref{f})$-$(\ref{outflow c})$  can be  established. Indeed, assume that $M_{+}>1$,  $u_{+}<0$ and a certain positive constant $\l>0$ hold.
 For a certain positive constant $\kappa \in (0, \lambda]$, there exists a small positive constant $\varepsilon_{0}>0$ such that if 	
$ e^{\frac{\lambda }{2} x}(n_{0}-\widetilde{n},\rho_{0} -\widetilde{\rho },u_{0}-\widetilde{u},n_{0}-\widetilde{n},v_{0}-\widetilde{v})   \in L^{2}(\mathbb{R_{+}})$
and
$$\| ( \rho_{0}-\widetilde{\rho}, u_{0}-\widetilde{u},n_{0}-\widetilde{n},v_{0}-\widetilde{v}))\|_{H^{1}} +\delta
	  \leq  \varepsilon_{0},$$
then the solution $(\rho, u, n, v)(t, x)$ to the IBVP $(\ref{f})$-$(\ref{outflow c})$ satisfies
	\begin{equation}
\| 	(\rho-\widetilde{\rho},	u-\widetilde{u},n-\widetilde{n}, v-\widetilde{v} ) (t)\|_{H^{1}}
	\leq 
	C_{6} \delta_{2}	e^{-\frac{\kappa_{1}}{2} t},
	\label{M_{+}>1 exp d}
	\end{equation}
	where $C_{6}>0$ and $\kappa_{1}\ll  \kappa$ are positive constants independent of time, and $\delta_{2}:=\| ( \rho_{0}-\widetilde{\rho},	u_{0}-\widetilde{u}, n_{0}-\widetilde{n},v_{0}-\widetilde{v})\|_{H^{1}}+  \|e^{\frac{\lambda}{2} x}( \rho_{0}-\widetilde{\rho},	u_{0}-\widetilde{u}, n_{0}-\widetilde{n},v_{0}-\widetilde{v})\|_{L^{2}} $ is a constant.
	
The proof of $(\ref{M_{+}>1 exp d})$ can be obtained by similar arguments as for $(\ref{M_{+}>1 algebra d})$. The details are omitted.
\end{rmk}
\begin{rmk}
In Theorem \ref{thm time decay},  we remove the restriction $(\ref{p small})$, and obtain  the nonlinear stability of steady-states and time decay rates of the solution to the  IBVP $(\ref{f})$-$(\ref{outflow c})$ for  $M_{+}=1$ with the weighted energy method. 
Moreover, if  $p'_{1}(\r_{+})=p'_{2}(n_{+})$, time decay rates of the solution to the IBVP $(\ref{f})$-$(\ref{outflow c})$ for  $M_{+}=1$ are the same as that of isentropic Navier-Stokes equations {\rm \cite{NNU}}.
\end{rmk}
We explain main strategies to prove  Theorems $\ref{thm stationary s}$-$\ref{thm time decay}$.
  The system $(\ref{f})$-$(\ref{outflow c})$ can be viewed as two compressible isentropic Navier-Stokes equations coupled with  each other through the drag force relaxation mechanisms.
Different from the  isentropic Navier-Stokes equations\cite{KNZ}, we can not reformulate 
   two momentum equations $(\ref{f})_{2}$ and $(\ref{f})_{4}$ into conservation forms due to the influence of  drag force, which implies that the steady-state satisfies the system $(\ref{stationary f})$-$(\ref{stationary boundary c})$ 
  consisting of a first-order and a second-order ordinary differential equations instead of  only a first-order  ordinary differential equation in \cite{KNZ},
   and it is not straightforward  to  apply the center manifold theory\cite{C}. 
   To overcome the difficulty, we  introduce a new 
   variable  $\wt{w}:=\wt{u}_{x}$, get the estimate
   $|\wt{u}_{x}(0)|\leq C|u_{-}-u_{+}|$  in Lemma \ref{lem u_{x}(0)},  and then rewrite the system $(\ref{stationary f})$ into the $3\times 3$   system $(\ref{sf bar})$ of autonomous    ordinary differential equations.    
   Since the condition  $p'_{1}(\r_{+})=p'_{2}(n_{+})$ is not necessary, we need subtle  analysis to obtain the sign of  ${\rm Re}\l_{i}, i=1,2,3$, where $\l_{i}, i=1,2,3$ are  three eigenvalues  of the linearized $3\times 3$ system of $(\ref{sf bar})$. 
%
It should be noticed that  the linearized  system of $(\ref{sf bar})$ has at least one eigenvalue with negative real part due to the effect of  drag force, so that we  obtain the existence of  steady-states  for the supersonic, sonic and  subsonic case in Theorem \ref{thm stationary s}, which is different from the outflow problem of the isentropic Navier-Stokes system\cite{KNZ}, where there is no steady-state for the subsonic case.

We establish the uniform estimates of the perturbation $(\h,\s,\b{\h},\b{\s}):=(\r-\wt{\r},u-\wt{u}, n-\wt{n},v-\wt{v})$ to prove the nonlinear stability of steady-states for the supersonic $M_{+}>1$, sonic case $M_{+}=1$  and subsonic case $M_{+}<1$.
  For $M_{+}=1$,  it is easy to check that $(\wt{\r}\wt{u}^{2}+p_{1}(\wt{\r}))_{x}$ and $(\wt{n}\wt{v}^{2}+p_{2}(\wt{n}))_{x}$ decay slower than $\wt{u}_{xx}$ for 
 $p'_{1}(\r_{+})\neq p'_{2}(n_{+})$ owing to $(\ref{stationary f})_{2}$, $(\ref{stationary f})_{4}$ and  $(\ref{sigma})$, which implies the  term
 \be
\int R_{2}dx:= \int  \h\s  \f{ (\wt{\r}\wt{u}^{2}+p_{1}(\wt{\r}))_{x}}{\wt{\r}}
  + \b{\h}\b{\s} \f{ (\wt{n}\wt{v}^{2}+p_{2}(\wt{n}))_{x}}{\wt{n}}  dx
  \label{cross t}
 \ee 
can not be controlled directly as in \cite{KNZ}.
 With the help of $\wt{u}_{x}\geq 0$ and $\wt{v}_{x}\geq 0$, we  turn to deal with the  terms 
 \be
\int R_{2}dx
 +\int (\r\s^{2} +p_{1}(\r) -p_{1}(\wt{\r}) -p'_{1}(\wt{\r})\h )\wt{u}_{x} dx
 +(n\b{\s}^{2} +p_{2}(n) -p_{2}(\wt{n}) -p'_{2}(\wt{n}) )\wt{v}_{x} dx,
 \ee 
the leading terms of which can be rewritten as two positive semidefinite 2 variable quadratic forms $(\ref{r_1'})$ under the condition  $(\ref{p small})$.
 
By the weighted energy method, we get the exponential or algebraic time decay rates for  the  supersonic case $M_{+}>1$ if the  initial perturbation belongs to the exponential or algebraic weighted Sobolev space, and obtain algebraic time decay rates for the sonic case $M_{+}=1$. 
To get the basic weighted energy estimates, we use $(\ref{M_{+}>1 stationary solution d})$ and the   dissipation on the relaxation friction term $\b{\s}-\s$, and decompose $\s$ as $\s=\b{\s} +(\s-\b{\s})$, as motivated by Li-Wang-Wang\cite{LWW}. 
Due to the  algebraic decay $(\ref{sigma})$ of steady-states,  convergence rates of steady-states  for the sonic case $M_{+}=1$ is worse than that of the supersonic case $M_{+}>1$. 
It is necessary to  use the delicate   algebraic decay $(\ref{sigma})$-$(\ref{sig})$ and the dissipation on the drag force $\b{\s}-\s$,  decompose $\s$ as $\s=\b{\s}+(\s-\b{\s})$, and obtain more delicate  estimates to get the convergence rates for $M_{+}=1$.
It should be noticed that we make full use of the dissipation on  drag force term and the viscous terms,  and take  a linear coordinate transformation
\be 
\begin{pmatrix}
	\h\\ \b{\h} \\ \b{\s}
\end{pmatrix}=
\boldsymbol{P}
\begin{pmatrix}
	\hat{\r}\\ \hat{n} \\ \hat{v}
\end{pmatrix},
\quad \quad \quad 
{\rm with ~a~ invertible~ matrix~ } \boldsymbol{P},
\ee 
to gain the 3 variable quadratic form $\hat{\l}_{1}\hat{\r}^{2}+\hat{\l}_{2}\hat{n}^{2}$ with $\hat{\l}_{1},\hat{\l}_{2}>0$, which plays an important role  in basic weighted energy estimates together with some crucial cancellations.
In fact,  we obtain the algebraic time decay rates of the solution to  IBVP $(\ref{f})$-$(\ref{outflow c})$ for $M_{+}=1$ for the initial perturbation satisfying $\sg^{-\f{\l}{2}}(\r_{0}-\wt{\r},u_{0}-\wt{u}, n_{0}-\wt{n}, v_{0}-\wt{v})\in L^{2}(\mb{R}_{+})$, with  $\l<\l^{*}, \l^{*}:=2+\sqrt{8+\f{1}{1+b^{2}}}$,  $b=\f{\r_{+}(u_{+}^{2}-p'_{1}(\r_{+}))}{|u_{+}|\sqrt{(\mu+n_{+})n_{+}}}$ and the  function $\sg\geq 0$ satisfying $(\ref{sig})$.  
This is  an interesting phenomena describing the influences of two fluids on each other somehow and it should be emphasized that $\l^{*}=5$ is the same as that of the  isentropic Navier-Stokes equations \cite{NNU} for $u_{+}^{2}=p'_{1}(\r_{+})=p'_{2}(n_{+})$.

\vspace{2ex}
\noindent{\textbf{Notation.}}      
We denote by $\| \cdot\|_{L^{p}}$ the norm of the usual Lebesgue space $L^{p}=L^{p}({\mathbb{R}_{+}})$,  $1 \leq p \leq \infty$. And if $p=2$, we write $\| \cdot \|_{L^{p}(\mathbb{R}_{+})}=\|\cdot\|$.
 $H^{s}(\mathbb{R}_{+})$ stands for the standard $s$-th  Sobolev space over $\mathbb{R}_{+}$ equipped with its norm
 \be \| f \|_{H^{s}(\mathbb{R}_{+})}=\| f \|_{s}:=( \sum_{i=0}^{s} \| \partial^{i}f \|^{2})^{\frac{1}{2}}. \nonumber\ee $C([0,T]; H^{1}(\mathbb{R}_{+}))$ represents the space of continuous functions on the interval $[0,T]$ with values in $H^{s}(\mathbb{R}_{+})$. 
 $L^{2}([0,T]; \mathcal{B}) $ denotes the space of $L^{2}$ functions on the interval $[0,T]$ with values in  Banach space $\mathcal{B}$. 
 For
 a scalar function $W(x)>0$, the  weighted $L^{2}(\mathbb{R}_{+})$ and $H^{1}(\mathbb{R}_{+})$ spaces are 
defined as follows:
\begin{equation}
\begin{aligned}
L^{2}_{W}(\mathbb{R}_{+}):=&
\left\{\begin{matrix}
~f \in L^{2}(\mathbb{R}_{+})~|
\end{matrix}\right.~\| f\|_{L^{2}_{W}}:=( \int_{\mathbb{R}_{+}} W (x)f^{2} dx )^{\frac{1}{2}}<+\infty 
\left.
\begin{matrix}~
\end{matrix}\right\},
\\
H^{1}_{W}(\mathbb{R}_{+}):=&
\left\{\begin{matrix}
~f \in H^{1}(\mathbb{R}_{+})~|
\end{matrix}\right.~\| f\|_{H^{1}_{W}}:=(\sum_{i=0}^{1} \|\partial^{i}f \|^{2}_{L^{2}_{W}} )^{\frac{1}{2}}<+\infty \left.\begin{matrix}~
\end{matrix}\right\}.
\nonumber
\end{aligned}
\end{equation}
 For a scalar function $W_{a,\nu}:=(1+x)^{\nu}$ with $\nu\geq0$, we denote $\| f\|_{a, \nu}:=\|  (1+x)^{\frac{\nu}{2}}f \|$.

The rest of this paper will be organized as follows. We prove the existence and uniqueness of  steady-states in Section \ref{sec:2}, get the nonlinear stability of  steady-states  in Section \ref{sec:$M_{+} >1$} for supersonic, sonic and subsonic cases, and obtain  convergence rates of  steady-states for the supersonic flow  in Subsection \ref{sec:$M'_{+} >1$} and the sonic flow in  Subsection  \ref{sec: $M'_{+}=1$}.
\section{Existence of Steady-State}
\label{sec:2}
We prove Theorem \ref{thm stationary s} on the  existence and uniqueness of  steady-states to the BVP $(\ref{stationary f})$-$(\ref{stationary boundary c})$  with	$u_{+}<0$ and $\delta$ sufficiently small as follows.
In order  to apply the center manifold theory \cite{C},   it is necessary 
to  get the  bounds  of $\wt{u}_{x}(0)$ or $\wt{v}_{x}(0)$. 
\begin{lem}
	\label{lem u_{x}(0)}
Assume that  $u_{+}<0$   and $\delta=|u_{-}-u_{+}|$ hold with $\delta$ sufficiently small. 	Then the steady-state  $(\wt{\r},\wt{u},\wt{n},\wt{v})$ to the  BVP $({\ref{stationary f}})$-$(\ref{stationary boundary c})$ satisfies
	\be
	|\wt{u}_{x}(0)|\leq C|u_{-}-u_{+}|, \quad|\wt{v}_{x}(0)|\leq C|u_{-}-u_{+}|,
	\ee
	 where $C>0$ is 	a positive constant.
\end{lem}
\begin{proof}
	Due to $\wt{\r}=\f{\r_{+}u_{+}}{\wt{u}}$ and $\wt{n}=\f{n_{+}u_{+}}{\wt{v}}$,  we have
	\be
	\left\lbrace 
	\begin{aligned}
		&(\r_{+}u_{+}\wt{u} +A_{1}\r^{\gm}_+ u^{\gm}_+ \wt{u}^{-\gm})_{x}
		=(\mu \wt{u}_{x})_{x}  +\f{n_+ u_{+}}{\wt{v}}(\wt{v}-\wt{u}),
		\\
		&
		(n_{+}u_{+}\wt{v} +A_{2}n^{\a}_{+}u^{\a}_{+} \wt{v}^{-\a})_{x}
		=(n_+ u_+ \f{\wt{v}_{x}}{\wt{v}})_x -\f{n_+ u_{+}}{\wt{v}}(\wt{v}-\wt{u}).
	\end{aligned}
	\right. 
	\label{sf}
	\ee
Adding  $(\ref{sf})_{1}$ to $(\ref{sf})_{2}$ and	integrating the resulted equation over $(0,+\infty)$ lead to
	\be 
	\mu \wt{u}_{x}(0)+\f{n_{+}u_{+}}{u_{-}}\wt{v}_{x}(0)=\f{1}{u_+}[(\r_{+}+n_{+}) u^{2}_+ - (A_{1}\gm \r^{\gm}_+ +A_{2} \a n^{\a}_+)] (u_{-} -u_{+})+O(|u_{-}-u_{+}|^{2}).
	\label{wt u v}
	\ee 
With the help of $(\ref{wt u v})$, we	multiply  $(\ref{sf})_{1}$ by $\wt{u}$, $(\ref{sf})_{2}$ by $\wt{v}$ respectively,  then integrate  the summation of the resulted equations over $(0, +\infty)$ to gain 
	\be 
	\begin{aligned}
		&
		\int_{0}^{+\infty} (\mu \wt{u}_{x}^{2}+n_+ u_+\f{\wt{v}^{2}_{x}}{\wt{v}})dx
		+\int_{0}^{+\infty}\f{n_+ u_+}{\wt{v}}(\wt{v}-\wt{u})^{2}dx
		\\
		=&
		-u_{-}[\mu \wt{u}_{x}(0)+n_{+}u_{+}\f{\wt{v}_{x}(0)}{u_{-}}]+
		[(\r_{+}+n_{+})u_{+}^{2}-(A_{1}\gm \r^{\gm}_{+}+A_{2}\a n^{\a}_{+} )](u_{-}-u_{+})+O(|u_{-}-u_{+}|^{2})
	\\
=	&
		O(|u_{-}-u_{+}|^{2}).
		\label{wt v-u}
	\end{aligned}
	\ee 
	Multiplying $(\ref{sf})_{2}$ by $\f{\wt{v}_{x}}{\wt{v}}$ and then integrating the resulted equation over $(0, \infty)$ yield 
	\be
	-n_+ u_+\f{\wt{v}^{2}_{x}(0)}{2u^{2}_{-}}+\int_{0}^{+\infty}A_{2}\a n^{\a}_{+}\f{u_{+}^{\a}}{\wt{v}^{\a+2}}\wt{v}^{2}_{x}dx
	= \int_{0}^{+\infty}\f{n_+ u_+}{\wt{v}}(\wt{v}-\wt{u})\f{\wt{v}_{x}}{\wt{v}}dx
	+\int_{0}^{+\infty}n_+ u_+\f{\wt{v}^{2}_{x}}{\wt{v}}dx
	\label{sf'b3}
	\ee
	We estimate terms in the right hand side of $(\ref{sf'b3})$. With $\inf\limits_{x\in\mb{R}_{+}} \wt{n}>0$, $\a\geq 1$ and  $(\ref{wt v-u})$, we have 
	\be 
	\begin{aligned}
	&
	\int_{0}^{+\infty}\f{n_+ u_+}{\wt{v}}(\wt{v}-\wt{u})\f{\wt{v}_{x}}{\wt{v}}dx
	+\int_{0}^{+\infty}n_+ u_+\f{\wt{v}^{2}_{x}}{\wt{v}}dx
	\\
	\leq&
	 C\int_{0}^{+\infty}\f{n_{+} u_+}{\wt{v}}(\wt{v}-\wt{u})^{2}dx+ \f{1}{4}\int_{0}^{+\infty}A_{2}\a n^{\a}_{+}\f{u_{+}^{\a}}{\wt{v}^{\a+2}}\wt{v}^{2}_{x}dx
	 \\
	\leq &
	C|u_{-}-u_{+}|^{2}+\f{1}{4}\int_{0}^{+\infty}A_{2}\a n^{\a}_{+}\f{u_{+}^{\a}}{\wt{v}^{\a+2}}\wt{v}^{2}_{x}dx.
		\label{sf''b2}
	\end{aligned}
	\ee 
	Combining $(\ref{sf'b3})$ and  $(\ref{sf''b2})$, we  get 
	\be
	|\wt{v}_{x}(0)|\leq C|u_{-}-u_{+}|, \quad |\wt{u}_{x}(0)|\leq C|u_{-}-u_{+}| .
	\label{sf1}
	\ee
\end{proof}
Then,  we can prove Theorem \ref{thm stationary s} with the above lemma.
For $(\ref{stationary f})_{2}$ and $(\ref{stationary f})_{4}$,  using $\wt{\r}=\f{\r_{+}u_{+}}{\wt{u}}$, $\wt{n}=\f{n_{+}u_{+}}{\wt{v}}$ and  integrating the summation of  $(\ref{stationary f})_{2}$ and $(\ref{stationary f})_{4}$ over $(x,+\infty)$,   we have
\begin{equation}
\left\lbrace 
\begin{split}
&\wt{v}_{x}=\f{\wt{v}}{n_{+}u_{+}}[\r_{+}u_{+}(\wt{u}-u_{+})+A_{1}\r^{\gm}_{+}(\f{u^{\gm}_{+}}{\wt{u}^{\gm}}-1)+n_{+}u_{+}(\wt{v}-u_{+})+A_{2}n^{\a}_{+}(\f{u^{\a}_{+}}{\wt{v}^{\a}}-1)-\mu \wt{u}_{x}],\\&
\wt{u}_{xx}=\f{1}{\mu}[(\r_{+}u_{+}-A_{1}\r^{\gm}_{+}u^{\gm}_{+}\wt{u}^{-\gm-1})\wt{u}_{x}-n_{+}u_{+}(1-\f{\wt{u}}{\wt{v}})].
\end{split}
\label{2sf}
\right. 
\end{equation}
Define $\wt{w}:=\wt{u}_{x}$ and  $\bar{U}:=(\bar{u},\bar{w},\bar{v})^{\rm T}:=(\wt{u}-u_{+},\wt{w},\wt{v}-u_{+})^{\rm T}$. The system $(\ref{2sf})$ can be  reformulated into the autonomous  system as follows
\be\left\lbrace 
\begin{aligned}&
	\b{U}_{x}=\boldsymbol{J_{+}}\b{U}  +(0, \b{g}_{2}(\b{U}),\b{g}_{3}(\b{U}))^{\rm T},
	\\
	&
	\b{U}_{-}:= (\bar{u},\bar{w}, \bar{v})^{\rm T}(0)
	=(u_{-}-u_{+},\wt{u}_{x}(0),u_{-}-u_{+})^{\rm T}, 
	~ \lim_{x\to\infty} \bar{U}^{\rm T}=(0,0,0),
\end{aligned}\right. 
\label{sf bar} \ee 
where
\begin{align}
\boldsymbol{J_{+}}=&
\begin{pmatrix}
	0 &1  &0 \\ 
	\f{n_{+}}{\mu }& \frac{\r_{+}u_{+}^{2}-A_{1}\gamma \rho^{\gamma}_{+}}{\mu u_{+}} &-\f{n_{+}}{\mu} \\ 
	\frac{\r_{+}u_{+}^{2}-A_{1}\gamma \rho^{\gamma}_{+}  }{n_{+}u_{+}}&-\f{\mu}{n_{+}}  & \frac{n_{+}u_{+}^{2}-A_{2}\alpha n^{\alpha}_{+} }{n_{+}u_{+}}
\end{pmatrix},
\end{align}
\begin{align}
\bar{g}_{2}(\b{U}) =&
\f{1}{2}(2\f{n_{+}}{\mu}\f{1}{u_{+}}\bar{v}^{2}-2\f{n_{+}}{\mu}\f{1}{u_{+}}\bar{u}\bar{v}+2\f{A_{1}\gm (\gm +1)\r^{\gm}_{+}}{\mu u^{2}_{+}}\bar{w}\bar{u})+O(|\b{U}|^{3}),
\label{g2}
\\
	\bar{g}_{3}(\b{U})	=&
	\f{1}{2}[ \f{A_{1}\gm(\gm+1)\r^{\gm}_{+}}{n_{+}u^{2}_{+}}\bar{u}^{2}
	+2\f{\r_{+}u_{+}^{2}-A_{1}\gm\r^{\gm}_{+}}{n_{+}u^{2}_{+}}\bar{v} \bar{u} +(2\f{n_{+}u_{+}^{2}-A_{2}\a n^{\a}_{+}}{n_{+}u^{2}_{+}}\notag
	\\
	&
	+\f{A_{2}\a(\a+1)n^{\a}_{+}}{n_{+}u^{2}_{+}})\bar{v}^{2}
	-2\f{\mu}{n_{+}u_{+}}\bar{w}\bar{v}]+O(|\b{U}|^{3}).
\label{g}
\end{align}
Three eigenvalues $\l_{1}, \l_{2}, \l_{3}$ of matrix $\boldsymbol{J_{+}}$ satisfy
\be\left\lbrace
\begin{aligned}
	&\l_{1}\l_{2}\l_{3}=-\f{(\r_{+}+n_{+})u^{2}_{+}-(A_{1}\gm \r^{\gm}_{+}+A_{2}\a n^{\a}_{+} )}{\mu u_{+}},
	\\&  
	\l_{1}+\l_{2}+\l_{3}=\f{\r_{+}u^{2}_{+}-A_{1}\gm \r^{\gm}_{+}}{\mu u_{+}}+\f{n_{+}u^{2}_{+}-A_{2}\a n^{\a}_{+}}{n_{+} u_{+}},
	\\&
	\l_{1}\l_{2}+\l_{1}\l_{3}+\l_{2}\l_{3}=\r_{+}\f{(u_{+}^{2}-A_{1}\gm\r_{+}^{\gm-1})(u_{+}^{2} -A_{2}\a n_{+}^{\a-1})}{\mu u^{2}_{+}}-1-\f{n_{+}}{\mu}.
\end{aligned}\right. \label{l f1}\ee 
If $M_{+}>1$, it is easy to obtain $\l_{1}\l_{2}\l_{3}>0$ and  $u_{+}^{2}>\min\{A_{1}\gm\r_{+}^{\gm-1},A_{2}\a n_{+}^{\a-1}\}$.
Without loss of generality, we assume $A_{1}\gm\r_{+}^{\gm-1}\geq A_{2}\a n_{+}^{\a-1}$. Moreover, we have 
\be
 \l_{1}+\l_{2}+\l_{3}<0 {~\rm for~}u_{+}^{2}>A_{1}\gm\r_{+}^{\gm-1},\quad 
 \l_{1}\l_{2}+\l_{1}\l_{3}+\l_{2}\l_{3}<0 {~\rm for~}  A_{2}\a n_{+}^{\a-1}\leq u^{2}_{+}\leq A_{1}\gm \r_{+}^{\gm-1},
 \ee
  which can imply  ${\rm Re}\l_{1}<0, ~{\rm Re}\l_{2}<0$ and $\l_{3}>0$ for $M_{+}>1$.
Using similar arguments, we have the following results:
\be 
\left\lbrace 
\begin{split}
	&{\rm if}~ M_{+}>1,{\rm then}~ {\rm Re}\l_{1}<0, {\rm Re}\l_{2}<0, \l_{3}>0,\\&
	{\rm if}~ M_{+}<1, {\rm then}~ {\rm Re}\l_{1}>0, {\rm Re}\l_{2}>0, \l_{3}<0, \\&
	{\rm if} ~M_{+}=1, {\rm then} ~\l_{1}>0, \l_{2}<0, \l_{3}=0. 
\end{split}
\right. 
\label{M_+}
\ee
Then, applying the center manifold theory \cite{C}, it is not difficult to show the supersonic or subsonic case $M_{+}\neq 1$ in Theorem \ref{thm stationary s} if $\delta$ is small.
Finally, we prove the sonic case $M_{+}=1$ in Theorem \ref{thm stationary s} which implies $\l_{1}>0, \l_{2}<0, \l_{3}=0$. 
The eigenvectors $r_{1}, r_{2},r_{3}$ of $\l_{1},\l_{2},\l_{3}$ are obtained  respectively as follows
\be
r_{1}=
\begin{pmatrix}
	1\\ 
	\l_{1}\\ 
	-\f{\mu}{n_{+}}(\l_{1}^{2}-\f{\r_{+}u^{2}_{+}-A_{1}\gm\r^{\gm}_{+} }{\mu u_{+}}\l_{1})+1
\end{pmatrix},
r_{2}=
\begin{pmatrix}
	1\\ 
	\l_{2}\\ 
	-\f{\mu}{n_{+}}(\l_{2}^{2}-\f{\r_{+}u^{2}_{+}-A_{1}\gm\r^{\gm}_{+} }{\mu u_{+}}\l_{2})+1
\end{pmatrix},
r_{3}=
\begin{pmatrix}
	1\\ 
	0\\ 
	1
\end{pmatrix}.
\label{ev}
\ee
Define the matrix $\boldsymbol{P_{1}}:=[r_{1},r_{2},r_{3}]$ and  take  a linear transformation $Z:=(z_{1},z_{2},z_{3})^{\rm T}=\boldsymbol{P_{1}}^{-1}\b{U}$.
With $(\ref{g2})$ and $(\ref{g})$,
the system $(\ref{sf bar})$  can be reformulated as follows
\be 
\left\lbrace 
\begin{aligned}
	&\frac{d}{dx}\begin{pmatrix}
		z_{1}\\ 
		z_{2}\\ 
		z_{3}
	\end{pmatrix}=\begin{pmatrix}
		\lambda_{1} & * &0 \\ 
		0& \lambda_{2} &0 \\ 
		0& 0 & \lambda_{3}
	\end{pmatrix}
	\begin{pmatrix}
		z_{1}\\ 
		z_{2}\\ 
		z_{3}
	\end{pmatrix}
	+
	\begin{pmatrix}
		g_{1}(z_{1},z_{2},z_{3})\\ 
		g_{2}(z_{1},z_{2},z_{3}) \\ 
		g_{3}(z_{1},z_{2},z_{3})
	\end{pmatrix},\\&
	(z_{1},z_{2},z_{3})(0)=(z_{1-},z_{2-},z_{3-})=(\boldsymbol{P_{1}}^{-1}\b{U}_{-})^{\rm T},\quad \lim_{x\to \infty}(z_{1},z_{2},z_{3})=(0,0,0),
	\label{sf z}
\end{aligned}
\right. 
\ee
where nonlinear functions $g_{i}(i=1,2,3)$  are denoted by
\be
\begin{pmatrix}
	g_{1}(z_{1},z_{2},z_{3})\\ 
	g_{2}(z_{1},z_{2},z_{3}) \\ 
	g_{3}(z_{1},z_{2},z_{3})
\end{pmatrix}
=\boldsymbol{P_{1}}^{-1}
\begin{pmatrix}
	0\\ 
	\bar{g}_{2}(\bar{u},\bar{w},\bar{v})\\ 
	\bar{g}_{3}(\bar{u},\bar{w},\bar{v})
\end{pmatrix}.
\label{bg}
\ee
With the help of  the  manifold theory \cite{C}, there exist a local center manifold $W^{c}(0,0,0)$ and a local stable manifold $W_{3}^{s}(0,0,0)$ 
\be
W^{c}(0,0,0)=\{(z_{1},z_{2},z_{3})~ |~ z_{1}=f^{c}_{1}(z_{3}), z_{2}=f^{c}_{2}(z_{3}),  |z_{3}|~{\rm sufficient ~small}   \},
\label{rd}
\ee
\be
W_{3}^{s}(0,0,0)=\{(z_{1},z_{2},z_{3})~ |~ z_{1}=f^{s}_{1}(z_{2}), z_{3}=f^{s}_{2}(z_{2}),  |z_{2}|~{\rm sufficient ~small}   \},
\ee 
where $f^{c}_{i}, f^{s}_{i}, i=1,2$ are smooth functions and $f^{c}_{i}(0)=0,~ Df^{c}_{i}(0)=0,~f^{s}_{i}(0)=0,~ Df^{s}_{i}(0)=0,~i=1,2$.
Using  $\bar{U}=PZ$, $(\ref{g3})$-$(\ref{g})$, and $(\ref{bg})$, we   gain
\be
\b{g}_{3}(z_{3})=az^{2}_{3}+O(|z_{1}|^{2}+|z_{2}|^{2}+|z_{3}|^{3}+|z_{1}z_{3}|+|z_{2}z_{3}|),
\label{g3}
\ee
where 
\be 
a=\f{A_{1}\gm(\gm+1) \r^{\gm}_{+}+A_{2}\a (\a+1)n^{\a}_{+}}{2u_{+}^{2}(1+b^{2})(\mu+n_{+})},\quad b:=\f{\r_{+}(u_{+}^{2}-p'_{1}(\r_{+}))}{|u_{+}|\sqrt{(\mu+n_{+})n_{+}}}.
\label{a}
\ee
Therefore, the system $(\ref{sf z})$  can be reformulated as follows
\be
\left\lbrace 
\begin{aligned}
	&z_{1x}=\l_{1}z_{1}+O(|Z|^{2}),
	\\
	&
	z_{2x}=\l_{2}z_{2}+O(|Z|^{2}),
	\\
	& 
	z_{3x}=az_{3}^{2}+O(|z_{1}|^{2}+|z_{2}|^{2}+|z_{3}|^{3}+|z_{1}z_{3}|+|z_{2}z_{3}|),
	\\
	&
	(z_{1},z_{2},z_{3})(0):=(z_{1-},z_{2-},z_{3-})=(\boldsymbol{P_{1}}^{-1}\b{U}_{-})^{\rm T}, ~ \lim_{x\to \infty}(z_{1},z_{2},z_{3})=(0,0,0).
	\label{z}
\end{aligned}
\right.
\ee
Let $\sigma_{1}(x)$ be a solution to $({\ref{z}})_{1}$ restricted on the local center manifold satisfying the equation 
\be 
\sigma_{1x}=a\sigma_{1}^{2}+O(\sigma_{1}^{3}),\quad \sigma_{1}(x) \to 0 ~{\rm as}~ x \to +\infty.
\label{sg}
\ee 
which implies  that  there exists the  monotonically increasing solution $\sigma_{1}(x)<0$ to $(\ref{sg})$   if  $\sigma_{1}(0)<0$ holds  and $|\sigma_{1}(0)|$  is sufficiently small.  
Therefore, if the initial data $(z_{1-},z_{2-},z_{3-})$ belongs to the region $\mathcal{M}\subset \mb{R}^{3}$ associated to the local stable manifold and the local center manifold, then  we have 
\be 
\left\lbrace 
\begin{aligned}
	&z_{i}=O(\sigma_{1}^{2})+O(\delta e^{-cx}),~i=1,2,
	\\&
	z_{3}=\sigma_{1}+O(\delta e^{-cx}),
\end{aligned}
\right. 
\ee 
with $z_{3-}<0$, the smallness of $|(z_{1-},z_{2-},z_{3-})|$ and 
\be
c\f{\delta}{1+\delta x}\leq |\sigma_{1}| \leq C\f{\delta}{1+\delta x},\quad   |\partial ^{k}\sigma_{1}|\leq C \f{\delta^{k+1}}{(1+\delta x)^{k+1}}, \quad C>0,~~k=0,1,2,3.
\ee  
Due to $\sg_{1}(x)\leq 0$, we define \be \sigma(x):=-\sigma_{1}, \ee
which satisfies
\be 
\sg_{x}=-a\sg^{2}+O(|\sg|^{3}),\quad  \sg\to 0 ~{\rm as}~x\to +\infty
\label{sg0}
\ee 

It is easy to get
\be 
|\partial_{x}^{k}(\wt{\r}-\r_{+},\wt{u}-u_{+}, \wt{n}-n_{+}, \wt{v}-u_{+})|\leq C  \f{\delta^{k+1}}{(1+\delta x)^{k+1}},\quad C>0,~~k=0,1,2,3,
\ee 
and
\be (\wt{u}-u_{+},\wt{v}-u_{+})=(-\sg(x),-\sg(x))+O(|\sg(x)|^{2}),\quad 
(\wt{u}_{x}, \wt{v}_{x})=(a\sigma^{2}(x), a\sigma^{2}(x) )+O(|\sigma|^{3}),
\ee
with the help of $\bar{U}=\boldsymbol{P}Z$ and $(\ref{ev})$. 
%

\section{Nonlinear stability of steady-states}

\label{sec:$M_{+} >1$}
%
The  function space $Y(0,T)$ for $T>0$ is denoted by
\begin{equation}
\begin{aligned}
Y(0,T):=\{
~(\h,\s,\b{\h},\b{\s})~ |~&
( \h,\s,\b{\h},\b{\s}) \in C([0,T]; H^{1}(\mb{R}_{+})), 
\\
&
( \h_{x},\b{\h}_{x}) \in L^{2}([0,T]; L^{2}(\mb{R}_{+})), ~(\psi_{x},\b{\s}_{x}) \in L^{2}([0,T]; H^{1}(\mb{R}_{+}))~
\}.
\end{aligned}
\end{equation}
Let
\be
\phi=\r-\wt{\r},\quad \psi=u-\wt{u},\quad \b{\h}=n-\wt{n},\quad \b{\s}=v-\wt{v}.
\ee 
Then the perturbation $(\phi,\psi,\b{\h},\b{\s})$ satisfies the following system
\be 
\left\lbrace 
\begin{aligned}
	&	\h_{t}+u\h_{x}+\r\s_{x}=-(\s\wt{\r}_{x}+\h\wt{u}_{x}),
	\\
	&
\s_{t}+u\s_{x}+\f{p'_{1}(\r)}{\r}\h_{x}-\f{\mu\s_{xx}}{\r}-\f{n(\b{\s}-\s)}{\r}
=F_{1},
\\
&
\b{\h}_{t}+v\b{\h}_{x}+n\b{\s}_{x}=-(\b{\s}\wt{n}_{x}+\b{\h}\wt{v}_{x}),
\\
&
\b{\s}_{t}+v\b{\s}_{x}+\f{p'_{2}(n)}{n}\b{\h}_{x}-\f{(n\b{\s}_{x})_{x}}{n}+(\b{\s}-\s)
	=F_{2},
	\label{f1}
\end{aligned}
\right. 
\ee 
where 
\be 
F_{1}=-[-\mu(\f{1}{\r}-\f{1}{\wt{\r}})\wt{u}_{xx}+\s\wt{u}_{x}+(\f{p'_{1}(\r)}{\r}-\f{p'_{1}(\wt{\r})}{\wt{\r}}) \wt{\r}_{x}-(\f{n}{\r}-\f{\wt{n}}{\wt{\r}})(\wt{v}-\wt{u})],
\label{F1}
\ee 
\be 
F_{2}=-[-(\f{1}{n}-\f{1}{\wt{n}})(\wt{n}\wt{v}_{x})_{x}-\f{(\b{\h}\wt{v}_{x})_{x}}{n}+\b{\s}\wt{v}_{x}+(\f{p'_{2}(n)}{n}-\f{p'_{2}(\wt{n})}{\wt{n}}) \wt{n}_{x}].~~~~~~~
\label{F2}
\ee 
The initial and boundary conditions to the system $(\ref{f1})$ satisfy
\be 
(\phi,\psi,\b{\h},\b{\s})(0,x):=(\h_{0},\s_{0},\b{\h}_{0},\b{\s}_{0})=(\r_{0}-\wt{\r},u_{0}-\wt{u},n_{0}-\wt{n},v_{0}-\wt{v}),
\label{intial d1}
\ee 
\be 
\lim_{x\to \infty}(\h_{0},\s_{0},\b{\h}_{0},\b{\s}_{0})=(0,0,0,0),\quad (\s,\b{\s})(t,0)=(0,0).
\label{boundary d1}
\ee 
\begin{prop}
	\label{prop time decay}
	Assume that the same assumptions  in Theorem $\ref{thm long time behavior}$ hold. Let $(\h,\s,\b{\h},\b{\s})$ be the solution to the problem $(\ref{f1})$-$(\ref{boundary d1})$ satisfying 
	$ (\h,\s,\b{\h},\b{\s}) \in Y(0,T)  $  for any time $T>0$. Then there exist positive constants $\varepsilon>0$ and $C>0$ independent of $T$ such that if 
	\be \sup_{0\leq t\leq T}\|  (\h,\s,\b{\h},\b{\s}) (t)\|_{1}+\delta\leq \v \label{priori e}\ee
	is satisfied, then  it holds for  arbitrary $t\in [0,T]$ that 	
	\begin{equation}
	\begin{split}
	&\| (\h,\s,\b{\h},\b{\s}) \|^{2}_{1}+\int_{0}^{t} \|  (\h_{x},\s_{x},\b{\h}_{x},\b{\s}_{x}) \|^{2}d\tau+\int_{0}^{t}\| (\b{\s}-\s,\s_{xx},\b{\s}_{xx})  \|^{2}d\tau \leq C\|(\h_{0},\s_{0},\b{\h}_{0},\b{\s}_{0}) \|_{1}^{2}.
	\label{e}
	\end{split}
	\end{equation}
\end{prop}  
With the help of  $(\ref{priori e})$, it is easy to verify  the following Sobolev inequality 
\begin{equation}
~~~~~\| (\h,\s,\b{\h},\b{\s})(t) \|_{L^{\infty}}\leq \|(\h,\s,\b{\h},\b{\s})(t) \|_{H^{1}}\leq  \sqrt{2}\v.
\label{infty 1}
\end{equation} 
\begin{lem}[\cite{KNZ} ]
	\label{lem d2}
For any function $\s(t,\cdot)\in H^{1}(\mb{R}_{+})$, it holds 
	\begin{align}   
\int_{0}^{\infty} 	\delta	e^{-c_{0}x}|\s|^{2}dx\leq &
C\delta(|\s(t,0)|^{2}+\| \s_{x}(t) \|^{2}),	
	\label{d'1}
\\
	\int_{0}^{\infty}\f{\delta^{j}}{(1+\delta x)^{j}}|\s|^{2}dx\leq &
	C\delta^{j-2}(|\s(t,0)|^{2}+\| \s_{x}(t) \|^{2}), \quad~ j>2,
	\label{d'2}
	\end{align} 
	where $\delta>0$, $c_{0}>0$ and $C>0$ are  positive constants. 
\end{lem}
	
With the  Lemma $\ref{lem d2}$, we can gain the basic $L^{2}$  energy estimates of $(\h,\s,\b{\h},\b{\s})$. 
\begin{lem}
	\label{lem e0}
	Under the same conditions in Proposition $\ref{prop time decay}$, then the solution $(\h,\s,\b{\h},\b{\s})$ to the problem $(\ref{f1})$-$(\ref{boundary d1})$ satisfies for $t \in[0,T]$ that
	\begin{equation}
	\begin{aligned}
	&
	\| (\h,\s,\b{\h},\b{\s}) \|^{2}+\int_{0}^{t}\|(\s_{x},\b{\s}_{x},\b{\s}-\s)  \|^{2}d\tau +\int_{0}^{t} |\h(t,0)|^{2}+|\b{\h}(t,0)|^{2}d\tau
	\\
\leq&
 C \| (\h_{0},\s_{0},\b{\h}_{0},\b{\s}_{0}) \|^{2}+C(\delta+\v)\int_{0}^{t}\| (\h_{x},\b{\h}_{x}) \|^{2}d\tau.
	\label{e0}
	\end{aligned}
	\end{equation}
\end{lem}
\begin{proof}
	Define 
	\be  
	\Phi_{1}(\r,\wt{\r})=\int_{\wt{\r}}^{\r}\f{p_{1}(s)-p_{1}(\wt{\r})}{s^{2}}ds,
	\quad  \mathcal{E}_{1}=\r(\f{\s^{2}}{2}+\Phi_{1}),  
	\label{mcE 1}
	\ee
	\be 
	 \Phi_{2}(n,\wt{n})=\int_{\wt{n}}^{n}\f{p_{2}(s)-p_{2}(\wt{n})}{s^{2}}ds,
	\quad \mathcal{E}_{2}=n(\f{\b{\s}^{2}}{2}+\Phi_{2}).
	\label{mcE 2}
	\ee
	Then, by $(\ref{f})$ and $(\ref{stationary f})$, the direct computations lead to 
	\be 
	\begin{split}
		&\quad (\mc{E}_{1}+\mc{E}_{2})_{t}+(G_{1}+G_{2})_{x}+n(\b{\s}-\s)^{2}+\mu \s_{x}^{2}+n\b{\s}^{2}_{x}
		+R_{1}+R_{2}=-R_{3},
	\end{split}
	\label{f_0}
	\ee 
	where
	\be \left\lbrace 
	\begin{aligned}
		&G_{1}:=u\mc{E}_{1}+v\mc{E}_{2} +(p_{1}(\r)-p_{1}(\wt{\r}))\s+(p_{2}(n)-p_{2}(\wt{n}))\b{\s}, 
		\\
		&
		G_{2}:=-(\mu \s\s_{x}+n\b{\s}\b{\s}_{x} +\b{\h}\b{\s}\wt{v}_{x}),
		\\
		&
		R_{1}:= [\r\s^{2} +p_{1}(\r) -p_{1}(\wt{\r}) -p'_{1}(\wt{\r}) \h]\wt{u}_{x} 
		+[n\b{\s}^{2} +p_{2}(n) -p_{2}(\wt{n}) -p'_{2}(\wt{n}) \b{\h}]\wt{v}_{x},
		\\
		& 
		R_{2}:=\h\s\f{\wt{\r}\wt{u}\wt{u}_{x} +(p_{1}(\wt{\r}))_{x}}{\wt{\r}}+\b{\h}\b{\s}\f{\wt{n}\wt{v}\wt{v}_{x} +(p_{2}(\wt{n}))_{x}}{\wt{n}},
		\\
		&
		R_{3}:= \b{\h}(\b{\s}-\s)(\wt{v} -\wt{u}) +\b{\h}\b{\s}_{x}\wt{v}_{x}.
	\end{aligned}
\right. 
\label{G}
	\ee 
	Integrating $(\ref{f_0})$ in $x$ over $\mb{R}_{+}$ leads to
	\be 
	\f{d}{dt}\int \mathcal{E}_{1}+\mathcal{E}_{2}dx-G_{1}(t,0)+\int n(\b{\s}-\s)^{2}+\mu \s_{x}^{2}+n\b{\s}^{2}_{x}dx
	+\int R_{1}dx+\int R_{2}dx=-\int R_{3}dx.
	\label{f_}
	\ee 
	 Under the condition $(\ref{boundary d1})$, we  get
		\be -G_{1}(t,0)=-u_{-}[\Phi_{1}(\r(t,0),\wt{\r}(0))+\Phi_{2}(n(t,0),\wt{n}(0))]\geq c(\h^{2}(t,0)+\b{\h}^{2}(t,0)).	\label{l_1}\ee 
	For the supersonic or subsonic case $M_{+}\neq1$, with the help of  $(\ref{M_{+}>1 stationary solution d})$,  $(\ref{infty 1})$ and $(\ref{d'1})$,   we have
	\be 
	\int_{0}^{\infty} |R_{1}|+|R_{2}|+|R_{3}| dx \leq C\delta \|(\h_{x},\s_{x},\b{\h}_{x},\b{\s}_{x},\b{\s}-\s) \|^{2}+C\delta(\h^{2}(t,0)+\b{\h}^{2}(t,0)).
	\label{r_1}
	\ee 
	For the sonic case  $M_{+}=1$ and the restriction $|p'_{1}({\r_{+}})-p'_{2}({n_{+}})|	\leq  \sqrt{2}|u_{+}|\min\{ (1+\f{\r_{+}}{n_{+}})[(\gm-1)p'_{1}(\r_{+})]^{\f{1}{2}},(1+\f{n_{+}}{\r_{+}})[(\a-1)p'_{2}(n_{+})]^{\f{1}{2}}  \}$,  using  $(\ref{sigma})$-$(\ref{sig})$, $(\ref{infty 1})$, and  $(\ref{d'2})$, we get
		\be 
\begin{aligned}
&
 \int R_{1}+R_{2}+R_{3}dx
 \\
\geq& \int (\s,\h)\boldsymbol{M_{1}}(\s,\h)^{\rm T}\wt{u}_{x}+(\b{\s},\b{\h})\boldsymbol{M_{2}}(\b{\s},\b{\h})^{\rm T}\wt{v}_{x}dx	-C\int \f{\delta^{3}}{(1+\delta x)^{3}}(\h^{2}+\s^{2}+\b{\h}^{2}+\b{\s}^{2})dx
 \\ &	-C\delta \int |\bar{\s}-\s|^{2}+\s_{x}^{2}dx 
		  -C\delta^{\f{1}{2}}\| (\h,\s,\b{\h},\b{\s})\| [\h^{2}(t,0)+\b{\h}^{2}(t,0)+\|(\h_{x},\s_{x},\b{\h}_{x},\b{\h}_{x})\|^{2}]
 \\
			\geq &
	-C(\delta+\v ) \|(\h_{x},\s_{x},\b{\h}_{x},\b{\s}_{x},\b{\s}-\s) \|^{2}-C\delta(\h^{2}(t,0)+\b{\h}^{2}(t,0)),
	\end{aligned}
\label{r_1'}
	\ee 
	where $\boldsymbol{M_{1}}$, $\boldsymbol{M_{2}}$ are positive definite or  non-negative definite   matrices defined by 
	\be 
	\begin{split}
		\boldsymbol{M_{1}}=
		\begin{pmatrix}
			&\r_{+}&\f{u_{+}^{2}-A_{1}\gm\r_{+}^{\gm-1}}{2u_{+}}	\\&
			\f{u_{+}^{2}-A_{1}\gm\r_{+}^{\gm-1}}{2u_{+}}&\f{A_{1}\gm (\gamma-1)\r_{+}^{\gm-2}}{2}
		\end{pmatrix}, \quad 
	\boldsymbol{M_{2}}=\begin{pmatrix}
			&n_{+}&\f{u_{+}^{2}-A_{2}\a n_{+}^{\a-1}}{2u_{+}}	\\&
			\f{u_{+}^{2}-A_{2}\a n_{+}^{\a-1}}{2u_{+}}& \f{A_{2}\a (\a-1)n_{+}^{\a-2}}{2}
		\end{pmatrix}.
	\end{split}
	\ee  
	Finally,	
	with the help of $(\ref{f_})$-$(\ref{r_1'})$,
	we  get $(\ref{e0})$.
	Hence,  the proof of Lemma \ref{lem e0} is completed.
\end{proof}

In order to complete the proof of Proposition $\ref{prop time decay}$,  we need to establish the high order  estimates of  $(\h,\s,\b{\h}, \b{\s})$.
\begin{lem}
	\label{lem e1}
Under the same conditions in Proposition $\ref{prop time decay}$, then the solution $(\h,\s,\b{\h},\b{\s})$ to the problem $(\ref{f1})$-$(\ref{boundary d1})$ satisfies  for $t \in[0,T]$ that
	\begin{equation}
	\begin{split}
	&\quad \| (\h_{x},\b{\h}_{x})  \|^{2}+\int_{0}^{t} \| (\h_{x},\b{\h}_{x}) \|^{2}d\tau +\int_{0}^{t}\h_{x}^{2}(t,0)+\b{\h}_{x}^{2}(t,0)d\tau\\&
		\leq C\| (\h_{0},\s_{0},\h_{0x},\b{\h}_{0},\b{\s}_{0},\b{\h}_{0x}) \|^{2}+C(\v+\delta) \int_{0}^{t}\|  (\s_{xx},\b{\s}_{xx})\|^{2}d\tau.
	\label{1'-order time e1}
	\end{split}
	\end{equation}
\end{lem}
\begin{proof}
	Differentiating $(\ref{f1})_{1}$ in $x$, then multiplying the resulted equation by $\mu \h_{x}$, $(\ref{f1})_{2}$ by $\wt{\r}^{2}\h_{x}$ respectively, we  gain
	\begin{align}
 & (\mu \f{\h_{x}^{2}}{2})_{t} +(\mu u\f{\h_{x}^{2}}{2})_{x} +\mu \wt{\r}\h_{x}\s_{xx}
 	\notag
	\\
	=&
	-\mu [\f{3}{2}\s_{x}\h_{x}^{2} +\h\h_{x}\s_{xx}+(\f{1}{2}\h_{x}\wt{u}_{x} +\s_{x}\wt{\r}_{x})\h_{x} +(\h\wt{u}_{x}+\s\wt{\r}_{x})_{x}\h_{x}],
	\label{h_{x}}
\\
	& (\wt{\r}^{2}\h_{x}\s)_{t}  -(\wt{\r}^{2}\h_{t}\s)_{x} +\wt{\r}^{2}\f{p'_{1}(\r)}{\r}\h_{x}^{2} -\mu \wt{\r}\h_{x}\s_{xx} 
		\notag
		\\
	=&
	-(\wt{\r}^{2}\h_{t}\s_{x}  +\wt{\r}^{2}u\h_{x}\s_{x}  +2\wt{\r}\wt{\r}_{x}\h_{t}\s )+	\mu\wt{\r}^{2}(\f{1}{\r} -\f{1}{\wt{\r}})\h_{x}\s_{xx}  +\wt{\r}^{2}\f{n}{\r}(\b{\s}-\s)\h_{x} +F_{1}\wt{\r}^{2}\h_{x}.
		\label{s_{x}}
\end{align}
	Similarly, differentiating $(\ref{f1})_{3}$ in $x$, then multiplying the resulted equation by $\b{\h}_{x}$, $(\ref{f1})_{4}$ by $\wt{n}\b{\h}_{x}$ respectively lead to
	\begin{align} 
&(\f{\b{\h}_{x}^{2}}{2})_{t} +(v\f{\b{\h}_{x}^{2}}{2})_{x} +\wt{n}\b{\h}_{x}\b{\s}_{xx}
=-[\f{3}{2}\b{\s}_{x}\b{\h}_{x}^{2}
 +\b{\h}\b{\h}_{x} \b{\s}_{xx} +(\f{1}{2}\b{\h}_{x}\wt{v}_{x} +\b{\s}_{x}\wt{n}_{x})\b{\h}_{x} -(\b{\h}\wt{v}_{x} +\b{\s}\wt{n}_{x})_{x}\b{\h}_{x}],
\\
&
  (\wt{n}\b{\h}_{x}\b{\s})_{t}-(\wt{n}\b{\h}_{t}\b{\s})_{x}
		+\wt{n}\f{p'_{2}(n)}{n}\b{\h}_{x}^{2}
		-\wt{n}\b{\h}_{x}\b{\s}_{xx}
		\notag
		\\
		=&
		-(\wt{n}\b{\h}_{t}\b{\s}_{x}
		+\wt{n}v \b{\h}_{x}\b{\s}_{x}  +\wt{n}_{x}\b{\h}_{t}\b{\s})
		+[\wt{n}\f{(\b{\h}\b{\s}_{x})_{x}}{n} +\wt{n}(\f{1}{n}-\f{1}{\wt{n}})(\wt{n}\b{\s}_{x})_{x}
		-\wt{n}(\b{\s}-\s)]\b{\h}_{x}
			\label{bs_{x}}
		\\
		&
		 +\wt{n}_{x}\b{\h}_{x}\b{\s}_{xx}+F_{2}\wt{n}\b{\h}_{x}.
		 \notag
	\end{align}
	Adding $(\ref{h_{x}})$-$(\ref{bs_{x}})$ together,   integrating the resulted equation over $\mb{R}_{+}$, we have
	\be 
	\begin{aligned} 
		& \f{d}{dt}\int (\mu \f{\h_{x}^{2}}{2} +\f{\b{\h}_{x}^{2}}{2} +\wt{\r}^{2}\h_{x}\s +\wt{n}\b{\h}_{x}\b{\s})dx +\int(\mu u\f{\h_{x}^{2}}{2} +v\f{\b{\h}_{x}^{2}}{2} -\wt{\r}^{2}\h_{t}\s -\wt{n}\b{\h}_{t}\b{\s})_{x}dx
		\\&
		 +\int (\wt{\r}^{2}\f{p'_{1}(\r)}{\r}\h_{x}^{2} +\wt{n}\f{p'_{2}(n)}{n}\b{\h}_{x}^{2})dx
		 \\
		=&
		\sum_{i=1}^{4}J_{i},
		\label{1f}
	\end{aligned}
	\ee 
	where
	\be 
	\begin{aligned}
		J_{1}=&
		-\int[ \wt{\r}^{2}(\h_{t} +u\h_{x})\s_{x} +\wt{n}(\b{\h}_{t} +v\b{\h}_{x})\b{\s}_{x} +2\wt{\r}\wt{\r}_{x}\h_{t}\s +\wt{n}_{x}\b{\h}_{t}\b{\s} ]dx,
		\\
		J_{2}=&
		\int -\mu \h\h_{x}\s_{xx} -\b{\h}\b{\h}_{x}\b{\s}_{xx} + \mu\wt{\r}^{2}(\f{1}{\r} -\f{1}{\wt{\r}})\h_{x}\s_{xx} +\wt{n}(\f{1}{n}-\f{1}{\wt{n}})(\wt{n}\b{\s}_{x})_{x}\b{\h}_{x} +\wt{\r}^{2}\f{n}{\r}(\b{\s} -\s)\h_{x} -\wt{n}(\b{\s} -\s)\b{\h}_{x}dx,
\\
		J_{3}=&
		-\int \f{3}{2}\mu \s_{x}\f{\h_{x}^{2}}{2}+\f{3}{2}\b{\s}_{x}\b{\h}_{x}^{2}dx+\int \wt{n}\f{(\b{\h}\b{\s}_{x})_{x}}{n}\b{\h}_{x}dx,\quad 
		\\
		J_{4}=&
		-\int [\mu (
		\f{1}{2}\h_{x}\wt{u}_{x}+\s_{x}\wt{\r}_{x})\h_{x}+\mu (\h\wt{u}_{x}+\s\wt{\r}_{x})_{x}\h_{x}
		-F_{1}\wt{\r}^{2}\h_{x}
		\\
		&
		+(\f{1}{2}\b{\h}_{x}\wt{v}_{x}+\b{\s}_{x}\wt{n}_{x})\b{\h}_{x}-(\b{\h}\wt{v}_{x}+\b{\s}\wt{n}_{x})_{x}\b{\h}_{x}-\wt{n}_{x}\b{\h}_{x}\b{\s}_{xx}
		-F_{2}\wt{n}\b{\h}_{x}]dx.
		\nonumber
	\end{aligned}
\ee
	First, we estimate terms in the left side of $(\ref{1f})$. 
	Under the condition $(\ref{boundary d1})$, the  terms in the left side is estimated as follows
	\begin{equation}
	\int (\mu u\f{\h_{x}^{2}}{2}+v\f{\b{\h}_{x}^{2}}{2}-\wt{\r}^{2}\h_{t}\s-\wt{n}\b{\h}_{t}\b{\s})_{x}dx=- u_{-}\f{\mu \h_{x}^{2}(0,t)+\b{\h}_{x}^{2}(0,t)}{2}\geq 0,
	\label{h(0)}
	\end{equation}
	\be 
		  \int \wt{\r}^{2}\f{p'_{1}(\r)}{\r}\h_{x}^{2}+\wt{n}\f{p'_{2}(n)}{n}\b{\h}_{x}^{2}dx
	\geq \r_{+}p'_{1}(\r_{+})\| \h_{x} \|^{2} 
	+p'_{2}(n_{+})\| \b{\h}_{x} \|^{2}-C(\v+\delta)\|(\h_{x},\b{\h}_{x}) \|^{2}.
	\ee 
 We turn to  estimate terms in the right hand side of $(\ref{1f})$. By  $(\ref{M_{+}>1 stationary solution d})$-$(\ref{sigma})$, $(\ref{f1})_{1}$, $(\ref{f1})_{3}$, $(\ref{d'1})$-$(\ref{d'2})$,   Cauchy-Schwartz inequality and Young inequality with $0<\e<1$, we  obtain 
	\begin{align}
		|J_{1}|
		\leq& C \|(\s_{x},\b{\s}_{x})\|^{2}+C\delta\|(\h_{x},\b{\h}_{x})\|^{2}+C\delta (\h^{2}(t,0)+\b{\h}^{2}(t,0)),
	\\
	|J_{2}|\leq &
	C\|(\h,\b{\h})\|_{L^{\infty}}\| (\h_{x},\b{\h}_{x},\s_{xx},\b{\s}_{xx}) \|^{2}+C_{\e}\|\b{\s}-\s \|^{2}+\eta\| (\h_{x},\b{\h}_{x}) \|^{2}+C\delta\| (\b{\h}_{x},\b{\s}_{x}) \|^{2}
	\notag
	\\
	\leq&
	 C(\v+\delta+\e)\| (\h_{x},\b{\h}_{x}) \|^{2}+C\v\| (\s_{xx},\b{\s}_{xx}) \|^{2}+C_{\e}\| \b{\s}-\s \|^{2}+C\delta\|\b{\s}_{x}  \|^{2},
   \\
	 |J_{3}|\leq&
	  C\| (\s_{x},\b{\s}_{x}) \|_{L^{\infty}} \| (\h_{x},\b{\h}_{x}) \|^{2}+C\| \b{\h} \|_{L^{\infty}}\| (\b{\h}_{x},\b{\s}_{xx}) \|^{2}
	  \notag
	  \\
    \leq &
    C\v \| (\h_{x},\s_{x},\b{\h}_{x},\b{\s}_{x}) \|^{2}+C\v\|( \s_{xx},\b{\s}_{xx}) \|^{2},
	\end{align} 
	\begin{align}
		|J_{4}|\leq&
		C\delta\| (\h_{x},\s_{x},\b{\h}_{x},\b{\s}_{x}) \|^{2}+C\delta\| \b{\s}_{xx} \|^{2}
		+C\delta(\h^{2}(t,0)+\b{\h}^{2}(t,0)).
		\label{I_{5}}
	\end{align} 
	Finally, the substitution of $(\ref{h(0)})$-$(\ref{I_{5}})$ into $(\ref{1f})$ for $\delta$, $\v$ and $\e$ small enough leads to 
	\begin{equation}
	\begin{aligned}
	&\f{d}{dt}\int(\h_{x}^{2}+\b{\h}_{x}^{2}+\wt{\r}^{2}\h_{x}\s +\wt{n}\b{\h}_{x}\b{\s})dx+\| (\h_{x},\b{\h}_{x})\|^{2}+\h_{x}^{2}(t,0)+\b{\h}_{x}^{2}(t,0)
	\\
	\leq&
	 C\| (\s_{x},\b{\s}_{x},\b{\s}_{x}-\s_{x}) \|^{2}+C(\delta+\v)\| (\s_{xx},\b{\s}_{xx}) \|^{2}+C\delta (\h^{2}(t,0)+\b{\h}^{2}(t,0)).
	\label{e1}
	\end{aligned}
	\end{equation}
	Integrating $(\ref{e1})$  in $\tau$
	over $[0,t]$, and using Lemma $\ref{lem e0}$ and Young inequality,  we have $(\ref{1'-order time e1})$.
\end{proof}

\begin{lem}
	\label{lem e1'}
Under the same conditions in Proposition $\ref{prop time decay}$, then the solution $(\h,\s,\b{\h},\b{\s})$ to the problem $(\ref{f1})$-$(\ref{boundary d1})$ satisfies  for $t \in[0,T]$ that
	\begin{equation}
	\begin{split}
	&
	~~~~  \|(\psi_{x},\b{\s}_{x} ) \|^{2} 	+\int_{0}^{t}\|( \psi_{xx},\b{\s}_{xx}) \|^{2} 	d\tau
	\leq C\|(\h_{0},\s_{0},\b{\h}_{0},\b{\s}_{0})\|^{2}_{1} . 
	\label{s 3}
	\end{split}
	\end{equation}
\end{lem}
\begin{proof}
	Multiplying $(\ref{f1})_{2}$ by $-\psi_{xx}$, $(\ref{f1})_{4}$ by $-\b{\s}_{xx}$, respectively,  then adding them together and integrating the resulted equation in $x$ over $\mathbb{R}_{+}$ 
	imply
	\begin{equation}
	\begin{split}
	&
	\frac{d}{dt} \int \frac{\psi^{2}_{x}}{2}+\f{\b{\s}_{x}^{2}}{2}dx  		+ \int  \mu \frac{1}{\rho}	\psi^{2}_{xx}+\b{\s}_{xx}^{2}  dx  =\sum^{3}_{i=1} K_{i},
	\label{psi_{xx} space w1}
	\end{split}
	\end{equation}
	where
	\begin{equation}
	\begin{aligned}
	K_{1}	=&
	\int [ 	u \psi_{x} \psi_{xx} 	-\f{n}{\r}(\b{\s}-\s)\s_{xx}+\f{p'_{1}(\r)}{\r}\h_{x}\s_{xx}+v \b{\s}_{x}\b{\s}_{xx}+(\b{\s}-\s)\b{\s}_{xx}	+\f{p'_{2}(n)}{n}\b{\h}_{x}\b{\s}_{xx}]		dx,
	\\
	K_{2} 	=&
	\int  [  \widetilde{u}_{x} \psi \psi_{xx}  	-\mu \widetilde{u}_{xx} ( \frac{1}{\rho}-\f{1}{\wt{\r}})\s_{xx}+(\f{p'_{1}(\r)}{\r}-\f{p'_{1}(\wt{\r})}{\wt{\r}})\s_{xx}\wt{\r}_{x}-(\f{n}{\r}-\f{\wt{n}}{\wt{\r}})(\wt{v}-\wt{u})\s_{xx} 	
	+\widetilde{v}_{x} \b{\s} \b{\s}_{xx}
	\\
	&	+ (\wt{n}\wt{v}_{x})_{x} ( \frac{1}{n}-\f{1}{\wt{n}})\b{\s}_{xx}+(\f{p'_{2}(n)}{n}-\f{p'_{2}(\wt{n})}{\wt{n}})\b{\s}_{xx}\wt{n}_{x}-\f{\wt{n}_{x}}{\wt{n}}\b{\s}_{x}\b{\s}_{xx} -\f{(\b{\h}\wt{v}_{x})_{x}}{n}\b{\s}_{xx}]	dx,
\\	
	K_{3} =&
	-\int  [\f{(\b{\h}\b{\s}_{x})_{x}}{n}\b{\s}_{xx}+(\f{1}{n}-\f{1}{\wt{n}})(\wt{n}\b{\s}_{x})_{x}\b{\s}_{xx}]	dx.
	\nonumber
	\end{aligned}
	\end{equation}
	 We estimate  terms in the left side of $(\ref{psi_{xx} space w1})$. 
By the decomposition $\frac{1}{\rho}=(\frac{1}{\rho}-\frac{1}{\widetilde{\rho}})+(\frac{1}{\widetilde{\rho}}-\frac{1} {  \rho_{+} } )+\frac{1} {  \rho_{+} } $,
	the second term  is estimated as follows:
	\begin{equation}
	\begin{aligned}
	\int	\frac{\mu}{\rho} 	\psi_{xx}^{2}+\b{\s}_{xx}^{2}	dx\geq &
	\f{\mu}{\r_{+}}\| \s_{xx} \|^{2} +\| \b{\s}_{xx} \|^{2} -C(\|\h \|_{L^{\infty}}+\delta)\| \s_{xx} \|^{2}
	\\
	 \geq&
	  \f{\mu}{\r_{+}}\| \s_{xx} \|^{2} +\| \b{\s}_{xx} \|^{2} -C(\v+\delta)\| \s_{xx} \|^{2}.
	\label{s_{xx} e}
	\end{aligned}
	\end{equation}
 We  turn to estimate  terms in the right side of $(\ref{psi_{xx} space w1})$. With the aid of $(\ref{M_{+}>1 stationary solution d})$, Sobolev inequality and Cauchy-Schwarz inequality,  we have
	\begin{align}
	| K_{1}|	\leq &
		\frac{\mu }{ 16\rho_{+} }	\| \psi_{xx} \|^{2}	+\f{1}{16}\|\b{\s}_{xx} \|^{2} +C \| (\h_{x},\s_{x},\b{\h}_{x},\b{\s}_{x},\b{\s}-\s) \|^{2},
	\label{K_{1} e }
   \\
	|K_{2}|	\leq &
	C   \delta  \| ( \h_{x},\psi_{x},\b{\h}_{x},\b{\s}_{x})  \|^{2} +C\delta\|(\s_{xx}, \b{\s}_{xx} )\|^{2} +C\delta(\h^{2}(t,0)+\b{\h}^{2}(t,0)),
	\label{K_{2} e }
	\\
	|K_{3} |	\leq &
		C 	\| \b{\h}\|_{L^{\infty}} \| (\b{\s}_{x},\b{\s}_{xx}) \|^{2} +C\| \b{\s}_{x} \|_{L^{\infty}} \| \b{\s}_{xx} \|~\| \b{\h}_{x} \|
	 \leq C\v\| (\b{\s}_{x},\b{\s}_{xx}) \|^{2}.
	\label{K_{3} e}
	\end{align}
	Finally, taking $\delta$ and $\varepsilon$ small enough and  substituting of $(\ref{s_{xx} e})$-$(\ref{K_{3} e})$ into $(\ref{psi_{xx} space w1})$, we  obtain 
	\begin{equation}
		\frac{d}{dt} 	\int	\psi_{x}^{2}+\b{\s}_{x}^{2} dx   	+	 \frac{\mu}{2\rho_{+}} \|\psi_{xx}\|^{2} +\f{1}{2}	\| \b{\s}_{xx}\|^{2}
	\leq 	
	C \| (\h_{x},\s_{x},\b{\h}_{x},\b{\s}_{x},\b{\s}-\s) \|^{2}	 +C\delta(\h^{2}(t,0)+\b{\h}^{2}(t,0)).
	\label{psi_{xx} e1'}
	\end{equation}
	Integrating  $(\ref{psi_{xx} e1'})$  in $\tau$ over $[0,t]$, and using  Lemmas $\ref{lem e0}$-$\ref{lem  e1}$ and the smallness of $\delta$ and $\varepsilon$, we obtain the desired estimate $(\ref{s 3})$.
	Therefore, we complete the proof of Lemma $\ref{lem e1'}$.
\end{proof}
With the help of  Lemmas \ref{lem e0}-\ref{lem e1'}, we get $(\ref{e})$ and complete the proof of Proposition \ref{prop time decay}.
\section{Time convergence rates}
\label{sec: convergence rate}

\subsection{ Convergence rate of supersonic steady-state} 
\label{sec:$M'_{+} >1$}


\begin{prop}
\label{prop M>1 time decay}
Assume that   the same conditions in Theorem $\ref{thm time decay}$  for $M_{+}>1$ hold and let $(\h,\s,\b{\h},\b{
\s})$ be a solution to the IBVP $(\ref{f1})$-$(\ref{boundary d1})$ satisfying  $(\h,\s,\b{\h},\b{\s})\in C([0,T]; H^{1})$ and 
${(1+x)^{\f{\nu}{2}}} (\h,\s,\b{\h},\b{\s}) \in C([0,T]; L^{2}) $  for any time  $T>0$. Then for arbitrary $\nu \in [0, \lambda]$, there exist positive constants $\varepsilon>0$ and $C>0$ independent of $T$ such that if 
\be
\sup_{0 \leq t \leq T}\| (\h,\s,\b{\h},\b{\s}) (t) \|_{1}+\delta\leq \v
\ee
is satisfied, it holds for arbitrary $t \in [0, T]$ that 	
\begin{equation}
\begin{aligned}
&
(1+t)^{\lambda  -\nu  +\theta} (\| (\h,\s,\b{\h},\b{\s}) \|_{1} +\| (\h,\s,\b{\h},\b{\s}) \|^{2}_{a,\nu})      +\nu \int_{0}^{t} (1 +\tau)^{\lambda  -\nu  + \theta} \|( \h,\b{\h},\s,\b{\s})\|^{2}_{a,\nu-1}d \tau
\\
&
+\int_{0}^{t} (1 + \tau)^{\lambda  -\nu  +\theta} \|( \s_{x},\b{\s}_{x},\b{\s}-\s) \|^{2}_{a,\nu}  d\tau +\int_{0}^{t} (1 +\tau)^{\lambda  -\nu  +\theta} \|( \h_{x},\s_{xx},\b{\h}_{x},\b{\s}_{xx}) \|^{2}  d\tau
\\
\leq &
C (1+t)^{\theta}(\| (\h_{0},\s_{0},\b{\h}_{0},\b{\s}_{0}) \|^{2}_{1} +\|(\h_{0},\s_{0},\b{\h}_{0},\b{\s}_{0})\|^{2}_{a,\l}) ,
\label{e__1}
\end{aligned}
\end{equation}
with  $\theta>0$. 
\end{prop}  
$~~~~$Our first goal is  to obtain the basic weighted energy estimates of $(\h,\s,\b{\h},\b{\s})$. 
\begin{lem}
\label{lem h e1}
Under the same conditions in Proposition $\ref{prop M>1 time decay}$, then the solution $(\h,\s,\b{\h},\b{\s})$ to the IBVP $(\ref{f1})$-$(\ref{boundary d1})$ satisfies for $t \in[0,T]$ that
\begin{equation}
\begin{aligned}
&
(1+t)^{\xi } \| (\h,\s,\b{\h},\b{\s}) \|^{2}_{a,\nu}   +\nu \int_{0}^{t} (1  +\tau )^{\xi } \| (\h,\s,\b{\h},\b{\s}) \|^{2}_{a,\nu-1}  d\tau
\\
&
+\int_{0}^{t}  (1+\tau)^{\xi}  \|(\s_{x},\b{\s}_{x},\b{\s}-\s)  \|^{2}_{a,\nu}  d \tau 
+\int_{0}^{t}(1+\tau)^{\xi}(\h^{2}(t,0)+\b{\h}^{2}(t,0))d\tau
\\
\leq&
 C\| (\h_{0},\s_{0},\b{\h}_{0},\b{\s}_{0}) \|^{2}_{a,\lambda}   +C \delta  \int_{0}^{t}  ( 1  + \tau )^{\xi}  \| (\h_{x},\b{\h}_{x}) \|^{2} d\tau
 \\
&+C\nu \int_{0}^{t}  ( 1  +\tau )^{\xi} \| (\s_{x},\b{\s}_{x},\b{\s}-\s) \|^{2}_{a,\nu-1}d\tau
+ C\xi  \int_{0}^{t}  (1  +\tau )^{\xi -1}  \| (\h,\s,\b{\h},\b{\s}) \|^{2}_{a,\nu} d\tau  .
\label{L^{2} time e1}
\end{aligned}
\end{equation}
with $\xi\geq0$.
\end{lem}
\begin{proof}
We multiply $(\ref{f_0})$ by $W_{a,\nu}$, where $W_{a,\nu}:=(1+x)^{\nu}$ is a space weight function. 
We integrate  the resulted equality over $\mathbb{R}_{+}$ to obtain
	\be 
\begin{aligned}
	&\f{d}{dt}\int W_{a,\nu}(\mc{E}_{1}+\mc{E}_{2})-(W_{a,\nu}G_{1})(t,0)-\int W_{a,\nu-1}G_{1}dx-\int W_{a,\nu-1}G_{2}dx
	\\
	&
	 +\int W_{a,\nu} [\mu \s_{x}^{2}+n\b{\s}^{2}_{x}+n(\b{\s}-\s)^{2}]dx
	 \\
	=&
	-\int W_{a,\nu}(R_{1}+R_{2}+R_{3})dx,
\end{aligned}
\label{f_01}
\ee 
where  $\mc{E}_{i}$, $i=1,2$ are defined by  $(\ref{mcE 1})$-$(\ref{mcE 2})$, and $G_{j}$ for  $j=1,2$, $R_{k}$ for $k=1,2,3$  are defined by $(\ref{G})$.\\
First, we estimate terms on the left hand side of (\ref{f_01}). Under the condition $(\ref{boundary d1})$, the second term on the left hand side is estimated as 
\begin{equation}
\begin{split}
&
~~~- (W_{a,\nu}
G_{1})(t,0)
=  |u_{-}|     [  \Phi_{1}(\r(t,0),\wt{\r}(0))+\Phi_{2}(n(t,0),\wt{n}(0)) ]   \geq c (\h^{2}(t,0)+\b{\h}^{2}(t,0)),~~~~~~~
\label{second term e}
\end{split}
\end{equation}
%
%
%
%
We decompose $\s$ as  	$\s=\b{\s}+(\s-\b{\s})$ 	and  use $(\ref{M_{+}>1 stationary solution d})$  to gain
\begin{equation}
\begin{aligned}
&
-\nu \int W_{a,\nu-1}G_{1}dx
\\
\geq&
 \nu\int W_{a,\nu-1}[\f{1}{2}(\h,\b{\h},\b{\s})\boldsymbol{M_{3}}(\h,\b{\h},\b{\s})^{\rm T} -\r_{+}u_{+}\b{\s}(\s-\b{\s}) -A_{1}\gm \r_{+}^{\gm-1}\h(\s-\b{\s})]dx
 \\
&-C(\delta+\v)\| (\h,\s,\b{\h},\b{\s}) \|^{2}_{a,\nu-1} ,
\label{G1}
\end{aligned} 
\end{equation}
where the symmetric matrix $\boldsymbol{M_{3}}$ is denoted by
\begin{equation}
\boldsymbol{M_{3}}=\begin{pmatrix}
-A_{1}\gm \r_{+}^{\gm-2}u_{+} &  0&-A_{1}\gm\r_{+}^{\gm-1}  \\ 
0 &  -A_{2}\a n_{+}^{\a-2} u_{+}  & -A_{2}\a n_{+}^{\a-1}\\ 
-A_{1}\gm\r_{+}^{\gm-1}  & -A_{2}\a n_{+}^{\a-1}&-(\r_{+}+n_{+})u_{+}
	\end{pmatrix}.	
\end{equation}
It is easy to verify that $\boldsymbol{M_{3}}$ is a positive  definite matrix  for $M_{+}>1$.
 Hence,  the estimate of the third term on the left hand side   is obtained under the condition $\varepsilon$, $\delta$ and $\e$ small enough that
\begin{equation} 
\begin{aligned} 
&
 -\nu \int  W_{a,\nu-1 }  G_{1} dx
 \\
\geq&
 	c \nu 	\| (\h,\s,\b{\h},\b{\s} ) \|^{2}_{a,\nu-1} -\e\nu\|(\h,\b{\s}) \|^{2}_{a,\nu-1}-C_{\e}\nu\|\b{\s}-\s \|^{2}_{a,\nu-1} -C(\v+\delta)	\| (\h,\s,\b{\h},\b{\s} ) \|^{2}_{a,\nu-1}
 	\\
\geq &	c \nu	\| (\h,\s,\b{\h},\b{\s} ) \|^{2}_{a,\nu-1}-C\nu \|\b{\s}-\s \|^{2}_{a,\nu-1},
\end{aligned}
\label{G1-1}
\end{equation} 
%
By Young inequality with $0<\e<1$,  the forth and  fifth terms on the left hand side are estimated as 
\be 
\begin{split}
&-\nu \int W_{a,\nu-1} G_{2}dx
\leq \nu \e \| (\s,\b{\s}) \|^{2}_{a,\nu-1}+\nu C_{\e}\| (\s_{x},\b{\s}_{x}) \|^{2}_{a,\nu-1}+C\delta\| (\b{\h},\b{\s}) \|^{2}_{a,\nu-1}.
\end{split}
\ee 
\begin{equation}
\begin{split}
&\quad \int W_{a,\nu}[\mu \s_{x}^{2}+n\b{\s}^{2}_{x}+n(\s-\b{\s})^{2}]dx
\geq 
c\| (\s_{x},\b{\s}_{x},\b{\s}-\s) \|^{2}_{a,\nu}
-C(\v+\delta)\| (\s_{x},\b{\s}_{x},\s-\b{\s}) \|^{2}_{a,\nu}.
\end{split}
\end{equation}
By $(\ref{M_{+}>1 stationary solution d})$ and $(\ref{d'1})$,  it follows from Sobolev inequality and  Cauchy-Schwarz inequality that
\begin{equation}
\begin{aligned}
	|\int W_{a,\nu}(R_{1}+R_{2}+R_{3})dx|
	\leq &
	C\delta \int e^{-\f{c_{0}}{2}x}(\h^{2}+\s^{2}+\b{\h}^{2}+\b{\s}^{2}+|\b{\s}-\s|^{2}) dx
	\\	
	 \leq&
	  C\delta\| (\h_{x},\s_{x},\b{\h}_{x},\b{\s}_{x},\b{\s}-\s) \|^{2}+C\delta(\h^{2}(t,0)+\b{\h}^{2}(t,0)).
	\label{mathcal{I}_{2} e}
	\end{aligned}
	\end{equation}
%
Finally, with  $\e$, $\delta$ and  $\varepsilon$ suitably small, the substitution of $(\ref{second term e})$-$(\ref{mathcal{I}_{2} e})$ into $(\ref{f_01})$
 leads to
\begin{equation}
\begin{aligned}
&
 \frac{d}{dt} \int  W_{a,\nu}	(\mc{E}_{1}+\mc{E}_{2})	dx
+c\nu\| (\h,\s,\b{\h},\b{\s})  \|^{2}_{a,\nu-1}   	+c 	\| (\psi_{x},\b{\s}_{x},\s-\b{\s}) \|^{2}_{a,\nu}+c(\h^{2}(t,0)+\b{\h}^{2}(t,0))
\\
\leq& 
C \delta 	\|(\h_{x},\b{\h}_{x}) \|^{2}+C\nu \| (\s_{x},\b{\s}_{x},\s-\b{\s}) \|^{2}_{a,\nu-1}.~~~~~~~
\label{w e0}
\end{aligned}
\end{equation}
Multiplying (\ref{w e0}) by $(1+\tau)^{\xi}$ and integrating the resulted equation in $\tau$ over $[0,t]$, we gain the desired estimate $(\ref{L^{2} time e1})$.
\end{proof}
Similar to Lemmas \ref{lem e1}-\ref{lem e1'}, we get the following high order  weighted  estimates of  $(\h,\s,\b{\h},\b{\s})$. The details are omitted.
\begin{lem}
	\label{lem h_{x} e1}
Under the same conditions in Proposition $\ref{prop M>1 time decay}$, then the solution $(\h,\s,\b{\h},\b{\s})$ to the IBVP $(\ref{f1})$-$(\ref{boundary d1})$ satisfies  for $t \in[0,T]$ that
		\begin{equation}
	\begin{aligned}
	&
 (1+t)^{\xi}	\| (\h_{x},\b{\h}_{x}) \|^{2}	+\int_{0}^{t}(1+\tau)^{\xi}\| (\h_{x},\b{\h}_{x}) \|^{2}d \tau
 \\	
\leq&
 C(\|(\h_{0},\s_{0},\b{\h}_{0},\b{\s}_{0} )\|^{2}_{a,\lambda}+\|(\h_{0x},\b{\h}_{0x}) \|^{2})
+C \nu	\int_{0}^{t} (1+\tau)^{\xi}\|  (\s_{x},\b{\s}_{x},\s-\b{\s}) \|^{2}_{a,\nu-1}d\tau
 \\
& 
+C \varepsilon  \int_{0}^{t}(1+\tau)^{\xi }	\| (\s_{xx},\b{\s}_{xx}) \|^{2}_{a,\nu}	d \tau 
+C \xi  \int_{0}^{t} (1+\tau)^{\xi-1}    (\| (\h,\s,\b{\h},\b{\s}) \|^{2}_{a,\nu}+\| (\h_{x},\b{\h}_{x}) \|^{2}  ) d\tau ,
\label{time e1-1}
\end{aligned}
\end{equation}
with $\xi\geq 0$.
\end{lem}
\begin{lem}\label{lem s_{x} e1}
Under the same conditions in Proposition $\ref{prop M>1 time decay}$ hold, then the solution $(\h,\s,\b{\h},\b{\s})$ to the IBVP $(\ref{f1})$-$(\ref{boundary d1})$ satisfies for $t \in[0,T]$ that
\begin{equation}
\begin{aligned}
&
(1+t)^{\xi}\| (\s_{x},\b{\s}_{x}) \|^{2}+\int_{0}^{t}(1+\tau)^{\xi}\| (\s_{xx},\b{\s}_{xx}) \|^{2}d \tau
\\
\leq&
C(\|(\h_{0},\s_{0},\b{\h}_{0},\b{\s}_{0} )\|^{2}_{a,\lambda}+\|(\h_{0x},\b{\h}_{0x},\s_{0x},\b{\s}_{0x}) \|^{2})
+C \nu	\int_{0}^{t} (1+\tau)^{\xi}\|  (\s_{x},\b{\s}_{x},\s-\b{\s}) \|^{2}_{a,\nu-1}d\tau
\\
&
+C \xi \int_{0}^{t} (1+\tau)^{\xi-1}    (\| (\h,\s,\b{\h},\b{\s}) \|^{2}_{a,\nu}+\| (\h_{x},\s_{x},\b{\h}_{x},\b{\s}_{x}) \|^{2}  ) d\tau
\label{time e1-2}
\end{aligned}
\end{equation}
with $\xi\geq 0$.
\end{lem}
\textit{\underline{Proof of Proposition \ref{prop M>1 time decay}}} 
For $\nu\in[0,\l]$ and $\xi\geq0$, it follows from  Lemmas $\ref{lem h e1}$-$\ref{lem s_{x} e1}$ that
\be \begin{aligned}
&
 (1+ t)^{\xi}(\| (\h,\s,\b{\h},\b{\s}) \|^{2}_{a,\nu}+\| (\h_{x},\s_{x},\b{\s}_{x},\b{\h}_{x}) \|^{2})+\nu \int(1+ t)^{\xi}\| (\h,\s,\b{\h},\b{\s}) \|^{2}_{a,\nu-1}d\tau
 \\
 &
 +\int_{0}^{t}(1+\tau)^{\xi}\| (\s_{x},\b{\s}_{x},\s-\b{\s})\|^{2}_{a,\nu}d\tau+\int_{0}^{t}(1+\tau)^{\xi}\|(\h_{x},\s_{xx},\b{\h}_{x},\b{\s}_{xx}) \|^{2}
d \tau
\\
\leq&
 C (\|(\h_{0},\s_{0},\b{\h}_{0},\b{\s}_{0}) \|^{2}_{a,\l}+\| (\h_{0x},\s_{0x},\b{\h}_{0x},\b{\s}_{0x}) \|^{2})+C\nu \int_{0}^{t}  ( 1  +\tau )^{\xi} \| (\s_{x},\b{\s}_{x},\b{\s}-\s) \|^{2}_{a,\nu-1}d\tau
 \\
&
 +C\xi\int_{0}^{t}(1+\tau)^{\xi-1}(\| (\h,\s,\b{\h},\b{\s}) \|^{2}_{a,\nu}+\| (\h_{x},\s_{x},\b{\s}_{x},\b{\h}_{x}) \|^{2}),
\label{e1-1}
\end{aligned}
\ee 
where $C>0$ is a generic positive constant independent of $T,\nu,$  and $ \xi$. Hence, applying similar   induction arguments as in  \cite{KM1,NM,CHS}  to (\ref{e1-1}), we  gain the desired estimate  (\ref{e__1}). 

Indeed, for any $\l>0$ and $k=0,1,2,...[\l]$, we have 
\be \begin{aligned}
	&
	(1+ t)^{k}(\| (\h,\s,\b{\h},\b{\s}) \|^{2}_{a,\l-k}+\| (\h_{x},\s_{x},\b{\s}_{x},\b{\h}_{x}) \|^{2})+\nu \int(1+ t)^{k}\| (\h,\s,\b{\h},\b{\s}) \|^{2}_{a,\l-k-1}d\tau
	\\
	&
	+\int_{0}^{t}(1+\tau)^{k}\| (\s_{x},\b{\s}_{x},\s-\b{\s})\|^{2}_{a,\l-k}d\tau+\int_{0}^{t}(1+\tau)^{k}\|(\h_{x},\s_{xx},\b{\h}_{x},\b{\s}_{xx}) \|^{2}
	d \tau
	\\
	\leq&
	C (\|(\h_{0},\s_{0},\b{\h}_{0},\b{\s}_{0}) \|^{2}_{a,\l}+\| (\h_{0x},\s_{0x},\b{\h}_{0x},\b{\s}_{0x}) \|^{2}),
	\label{e1-2}
\end{aligned}
\ee 
and 
\be \begin{aligned}
	&
	(1+ t)^{k}\| (\h,\s,\b{\h},\b{\s}) \|^{2}_{1}
	+\int_{0}^{t}(1+\tau)^{k}\| (\h_{x},\s_{x},\s_{xx},\b{\s}_{x},\b{\h}_{x},\b{\s}_{xx},\s-\b{\s})\|^{2}d\tau
	\\
	\leq&
	C (\|(\h_{0},\s_{0},\b{\h}_{0},\b{\s}_{0}) \|^{2}_{a,\l}+\| (\h_{0x},\s_{0x},\b{\h}_{0x},\b{\s}_{0x}) \|^{2}).
	\label{e1-3}
\end{aligned}
\ee 
To prove $(\ref{e1-2})$ and $(\ref{e1-3})$, we apply similar induction arguments   as in  \cite{KM1,NM,CHS}  to (\ref{e1-1}).

Step 1. Taking $\xi=0$, $\nu=\l$ in $(\ref{e1-1})$ and using $(\ref{e})$, we have $(\ref{e1-2})$ and $(\ref{e1-3})$ for $k=0$. Therefore, $(\ref{e1-2})$ and $(\ref{e1-3})$ hold for $0<\l<1$.

Step 2. Taking $\xi=1$, $\nu=0$ in $(\ref{e1-1})$ and using $(\ref{e1-2})$ with $k=0$, we have $(\ref{e1-3})$ with $k=1$. Then, taking $\xi=1$, $\nu=\l-1$ in $(\ref{e1-1})$ and using  $(\ref{e1-3})$  with $k=1$ and $(\ref{e1-2})$ with $k=0$, we obtain the desired estimate $(\ref{e1-2})$ with $k=1$. Therefore, the proof is finished for $1\leq\l<2$.

Step 3. We repeat the same procedure as in Step 2. The estimate $(\ref{e1-1})$ (with $\xi=2$, $\nu=0$) together with $(\ref{e1-3})$ (with $k=1$) lead to $(\ref{e1-3})$ (with $k=2$). Also, $(\ref{e1-1})$ (with $\xi=2$, $\nu=\l-2$) together with $(\ref{e1-3})$ (with $k=2$) and $(\ref{e1-2})$ (with $k=1$) lead to $(\ref{e1-2})$ (with $k=2$), which proves the estimates $(\ref{e1-2})$ and $(\ref{e1-3})$ for $2\leq\l<3$. 

Repeating the same procedure, we get the desired estimates $(\ref{e1-2})$ and $(\ref{e1-3})$ for any $\l>0$.

If $\l>0$ is integer, we obtain $(\ref{e__1})$  from $(\ref{e1-2})$ letting $k=\l$. 

If $\l>0$ is not  integer, we obtain $(\ref{e__1})$ as follows. \\
Taking $\nu=0$ in $(\ref{e1-1})$, we have 
\be \begin{aligned}
	&
	(1+ t)^{\xi}\| (\h,\s,\b{\h},\b{\s}) \|^{2}_{1}
	+\int_{0}^{t}(1+\tau)^{\xi}\| (\h_{x},\s_{x},\s_{xx},\b{\h}_{x},\b{\s}_{x},\b{\s}_{xx},\s-\b{\s})\|^{2}d\tau
	\\
	\leq&
	C (\|(\h_{0},\s_{0},\b{\h}_{0},\b{\s}_{0}) \|^{2}_{a,\l}+\| (\h_{0x},\s_{0x},\b{\h}_{0x},\b{\s}_{0x}) \|^{2})
	+C\xi\int_{0}^{t}(1+\tau)^{\xi-1}\| (\h,\s,\b{\h},\b{\s}) \|^{2}_{1}d\tau.
	\label{e1-4}
\end{aligned}
\ee 
Using $(\ref{e1-2})$ with $k=[\l]$ and taking $s=1-(\l-[\l])$, we have 
\begin{align}
	&\int_{0}^{t}(1+\tau)^{\xi-1}\| (\h,\s,\b{\h},\b{\s}) \|^{2}_{1}d\tau
	\notag
	\\
	\leq &
	 \int_{0}^{t}(1+\tau)^{\xi-1-[\l]}\{
	  ( (1+t)^{[\l]}\| (\h,\s,\b{\h},\b{\h}) \|^{2}_{a,\l-[\l]})^{s} ( (1+t)^{[\l]}\| (\h,\s,\b{\h},\b{\h}) \|^{2}_{a,\l-[\l]-1})^{1-s}
	  \notag
	  \\
	  &
	  +
(1+t)^{[\l]}\|(\h_{x},\s_{x},\b{\h}_{x},\b{\s}_{x})  \|^{2}\}d\tau 
\notag
	 \\
	 \leq &
	  C (\|(\h_{0},\s_{0},\b{\h}_{0},\b{\s}_{0}) \|^{2}_{a,\l}+\| (\h_{0x},\s_{0x},\b{\h}_{0x},\b{\s}_{0x}) \|^{2})^{s}
	  \int_{0}^{t}(1+\tau)^{\xi-1-[\l]} ( (1+t)^{[\l]} \|(\h_{x},\s_{x},\b{\h}_{x},\b{\s}_{x})\|^{2}
	  \notag
	  \\
	  &
	  +(1+t)^{[\l]} \| (\h,\s,\b{\h},\b{\h}) \|^{2}_{a,\l-[\l]-1})^{1-s}d\tau 
	  \notag
	  \end{align}
	  \begin{align} 
	  \leq&
	   C (\|(\h_{0},\s_{0},\b{\h}_{0},\b{\s}_{0}) \|^{2}_{a,\l}+\| (\h_{0x},\s_{0x},\b{\h}_{0x},\b{\s}_{0x}) \|^{2}) (\int_{0}^{t}(1+\tau)^{\f{\xi-1-[\l]}{1+[\l]-\l}}d\tau)^{1+[\l]-\l}
	   \notag
	   \\
	   \leq &
	   C (\|(\h_{0},\s_{0},\b{\h}_{0},\b{\s}_{0}) \|^{2}_{a,\l}+\| (\h_{0x},\s_{0x},\b{\h}_{0x},\b{\s}_{0x}) \|^{2})(1+t)^{\theta},
	\label{e1-5}
\end{align}
where we take $\xi=\l+\theta(1+[\l]-\l)$ and $\theta>0$. 
\subsection{Convergence rate of  sonic steady-state}
\label{sec: $M'_{+}=1$}
The  function space $Y_{W}(0,T)$ for $T>0$ is denoted by
\begin{equation}
\begin{aligned}
Y_{W}(0,T):=\{
~(\h,\s,\b{\h},\b{\s})~ |~&( \h,\s,\b{\h},\b{\s}) \in C( [0,T]; H_{W}^{1}(\mb{R}_{+}) ), 
\\
&
( \h_{x},\b{\h}_{x}) \in L^{2}([0,T]; L_{W}^{2}(\mb{R}_{+}) ), ~(\psi_{x},\b{\s}_{x}) \in L^{2}([0,T]; H_{W}^{1}(\mb{R}_{+}) )~
\}.
\end{aligned}
\end{equation}
\begin{prop}
\label{prop M_{+}=1 time decay }
Assume that $1\leq \l<\l^{*}$ with $\l^{*}:=2+\sqrt{8+\f{1}{1+b^{2}}}$,
$b:=\f{\r_{+}(u_{+}^{2}-p'_{1}(\r_{+}))}{|u_{+}|\sqrt{(\mu+n_{+})n_{+}}}$,  and that the same conditions in Theorem $\ref{thm time decay}$   hold for $M_{+}=1$. Let $(\h,\s,\b{\h},\b{
	\s})$ be a solution to the IBVP $(\ref{f1})$-$(\ref{boundary d1})$ satisfying 
$ (\h,\s,\b{\h},\b{\s}) \in Y_{\sg^{-\l}}(0,T)  $  for any time $T>0$. Then for  arbitrary $\nu \in (0, \lambda]$, there exist positive constants $\varepsilon>0$ and $C>0$ independent of $T$ such that if 
\be 
\sup_{0 \leq t \leq T} \|\sg^{-\f{\l}{2}} (\h,\s,\b{\h},\b{\s})(t)  \|_{1}+\delta^{\f{1}{2}}\leq \v 
\label{M_{+}=1 prior a}
\ee 
is satisfied, it holds for arbitrary $t \in [0, T]$ that 	
\begin{equation}
\begin{aligned}
&
( 1    +\delta t )^{\frac{\lambda    -\nu}{2}   +\beta}	\| \sg^{-\f{\nu}{2}}(\h,\s,\b{\h},\b{\s}) \|^{2}_{1}  	+	\int^{t}_{0}	( 1   +\delta \tau )^{  \frac{\lambda-\nu}{2} +\beta}	\| \sg^{-\f{\nu-2}{2}}(\h,\s,\b{\h},\b{\s}) \|^{2}d\tau
\\
&
+	\int^{t}_{0}  ( 1  +\delta \tau  )^{  \frac{\lambda-\nu }{2} +\beta}	\| \sg^{-\f{\nu}{2}}( \h_{x},\s_{x}, \b{\h}_{x},\b{\s}_{x}) \|^{2}   d\tau 
+	\int^{t}_{0}  ( 1  +\delta \tau  )^{  \frac{\lambda-\nu }{2} +\beta}	\| \sg^{-\f{\nu}{2}}( \s_{xx}, \b{\s}_{xx},\b{\s}-\s) \|^{2}   d\tau
\\
\leq &
C(  1   +\delta t )^{ \beta} 	\| \sg^{-\f{\l}{2}}(\h_{0},\s_{0},\h_{0x},\s_{0x},\b{\h}_{0},\b{\s}_{0},\b{\h}_{0x},\b{\s}_{0x}
)\|^{2},~~~~~~~~~~~~~~~~~~~~~~~~~~~~~~~~~~~~~~~~~
\end{aligned}
\end{equation}
with  $\beta>0$.
\end{prop}
By the fact $\l\geq 1$ and  $(\ref{M_{+}=1 prior a})$, it is easy to verify the following estimate:
\begin{equation}
\| \sg^{-\f{1}{2}}(\h,\s,\b{\h},\b{\s})  \|_{L^{\infty}}  \leq \| \sg^{-\f{\l}{2}}(\h,\s,\b{\h},\b{\s}) \|_{1}\leq  \sqrt{2}\v.
\label{infty e2}
\end{equation} 

To deal with some nonlinear terms,  we use the following inequality as in \cite{NN,KNNZ out,Yin}. 
\begin{lem}[\cite{Yin} ]
	\label{lem non f}
	Let $\nu\geq1 $. Then a function $\sg^{-\f{\nu}{2}}(x)\h(t,x)\in H^{1}(\mb{R}_{+})$ satisfies 
	\be 
	\int \sg^{-\f{\nu-1}{2}}|\h|^{3}dx\leq C\| \sg^{-\f{1}{2}}\h \| ~(\sg(0)\h^{2}(t,0)+\| \sg^{-\f{\nu}{2}}\h_{x} \|^{2}+\| \sg^{-\f{\nu-2}{2}} \h \|^{2}),
	\label{nonlinear f}
	\ee  
	where the function $\sg(x)\geq 0$ is defined by $(\ref{sg0})$ with $\sg(0)$ small enough.
\end{lem}
To gain faster decay rates, it is necessary to use the following Hardy type inequality. 
\begin{lem}[\cite{KK} ]
	\label{hardy lem}
	Let $\zeta \in C^{1}[0,\infty)$ satisfies $\zeta>0$, $\zeta_{x}>0$ and $\zeta(x) \rightarrow \infty$ for $x\rightarrow \infty$. Then we have 
	\begin{equation}
	\int_{\mathbb{R}_{+}} \psi^{2} \zeta_{x} dx \leq 4 \int_{\mathbb{R}_{+}} \psi^{2}_{x} \frac{\zeta^{2}}{\zeta_{x}}dx
	\label{hardy ineq}
	\end{equation}	
	for $\s$ satisfying $\s(t,0)=0$ and $\sqrt{w} \psi \in H^{1 }(\mb{R}_{+})$, with the function $w:=\frac{\zeta^{2}}{\zeta_{x}}$. 
\end{lem}


With the aid of Lemmas  $\ref{lem non f}$-$\ref{hardy lem}$, we obtain the weighted $L^{2}$ estimate of $(\h,\s,\b{\h},\b{\s})$.
\begin{lem}
	\label{lem L^{2} e2}
Under the same conditions in Proposition $\ref{prop M_{+}=1 time decay }$, then the solution $(\h,\s,\b{\h},\b{\s})$ to the problem $(\ref{f1})$-$(\ref{boundary d1})$ satisfies  for $t \in[0,T]$ that
\begin{equation}
\begin{aligned}
&
(1+\delta  \tau)^{\xi} \|\sg^{-\f{\nu}{2}}(\h,\s,\b{\h},\b{\s})\|^{2}	+\int^{t}_{0} (1  +\delta \tau)^{\xi}	\|\sg^{-\f{\nu-2}{2}} (\h,\s,\b{\h},\b{\s} )\|^{2}	d\tau 
\\
&
  +\int^{t}_{0}	(1+\delta \tau)^{\xi} 	 \|\sg^{-\f{\nu}{2}}(\s_{x},\b{\s}_{x},\b{\s}-\s)\|^{2}d\tau+\int^{t}_{0}	(1+\delta \tau)^{\xi} 	\f{1}{\delta^{\nu}}(\h^{2}(t,0)+\b{\h}^{2}(t,0))d\tau 
\\
\leq&
C\| \sg^{-\f{\l}{2}}(\h_{0},\s_{0},\b{\h}_{0},\b{\s}_{0})  \|^{2}  +C\delta 	\int^{t}_{0} 	(1   +\delta  \tau)^{\xi} \|\sg^{-\f{\nu}{2}} (\h_{x},\b{\h}_{x})\|^{2}	d\tau
\\
& 
+C\delta \xi	\int^{t}_{0}	(1  +\delta  \tau)^{\xi  -1}	\|\sg^{-\f{\nu}{2}} (\h,\s,\b{\h},\b{\s}) \|^{2}d\tau, ~~~~~~~~~~~~~~~~~~~~~~~~~~~~~~~~~~~~~~~~~~          
\label{L^{2} e2}
\end{aligned} 
\end{equation}
with $\xi\geq0$.
\end{lem}	
\begin{proof}
We multiply $(\ref{f_0})$ by the space weight function $\sg^{-\nu}$, where the space weight function $\sg\geq 0$  satisfies  $(\ref{sig})$ and $(\ref{sg0})$. Then, we integrate the resulted equation over $\mb{R}_{+}$ to get
	\be 
\begin{aligned}
	&
	 \f{d}{dt}\int \sg^{-\nu}(\mc{E}_{1}+\mc{E}_{2})dx-(\sg^{-\nu}G_{1})(t,0)-a\nu \int \sg^{-(\nu-1)}G_{1}dx-a\nu \int \sg^{-(\nu-1)}G_{2}dx
	\\
	&
	+\int \sg^{-\nu }n(\b{\s}-\s)^{2}dx +\int \sg^{-\nu}(\mu \s_{x}^{2}+n\b{\s}^{2}_{x})dx+\int \sg^{-\nu}R_{1} dx
	\\
=&
-\int\sg^{-\nu} R_{2}dx-\int \sg^{-\nu}R_{3}dx,
\end{aligned}
\label{f_02}
\ee 
where  $\mc{E}_{i}$, $i=1,2$ are defined by  $(\ref{mcE 1})$-$(\ref{mcE 2})$, and $G_{j}$ for  $j=1,2$, $R_{k}$ for $k=1,2,3$  are defined by $(\ref{G})$.
First, we estimate  terms  on the left  hand side of $(\ref{f_02})$.
Under the condition $(\ref{boundary d1})$,  the second term on the left hand side  is estimated as 
\begin{equation}
\begin{split}
&
~~~- (\sg^{-\nu}
G_{1})(t,0)
 \geq \f{c}{\delta^{\nu}} (\h^{2}(t,0)+\b{\h}^{2}(t,0))\geq0.~~~~~~~
\label{second term e2}
\end{split}
\end{equation}
%
%
For the third term on the left hand side, using $(\ref{nonlinear f})$
and $\s=\b{\s}+(\s-\b{\s})$ yields
\begin{equation}
\begin{aligned}
&
-a\nu \int \sg^{-(\nu-1)}G_{1}dx
\\
\geq&
 a\nu \int \sg^{-(\nu-1)}[\f{1}{2} (\h,\b{\h},\b{\s})\boldsymbol{M_{4}}(\h,\b{\h},\b{\s})^{\rm T}
 -\r_{+}u_{+}\b{\s}(\s-\b{\s})-A_{1}\gm\r^{\gm-1}_{+}\h(\s-\b{\s})]dx
 \\
 &
 + a\nu \int\sg^{-(\nu-1)}
[-(A_{1}\gm\wt{\r}^{\gm-2}\wt{u}-A_{1}\gm\wt{\r}_{+}^{\gm-2}u_{+})\f{\h^{2}}{2}
-(A_{2}\a \wt{n}^{\a-2}\wt{v}-A_{2}\a n_{+}^{\a-2}u_{+})\f{\b{\h}^{2}}{2}
\\
&
-(A_{1}\gm\wt{\r}^{\gm-1}-A_{1}\gm\r_{+}^{\gm-1})\h\s
-(A_{2}\a\wt{n}^{\a-1}-A_{2}\a n_{+}^{\a-1})\b{\h}\b{\s}]dx-C\f{\v}{\delta^{\nu}}(\h^{2}(t,0)+\b{\h}^{2}(t,0))
\\
&
-C\v(\|\sg^{-(\f{\nu-2}{2})}(\h,\s,\b{\h},\b{\s})\|^{2}+\|\sg^{-\f{\nu}{2}}(\h_{x},\s_{x},\b{\h}_{x},\b{\s}_{x})\|^{2}),
\label{bound}
\end{aligned}
\end{equation}
where the symmetric matrix  $\boldsymbol{M_{4}}$ is defined as 
\begin{equation}
\begin{split}
&\boldsymbol{M_{4}}=
\begin{pmatrix}
-A_{1}\gm \r_{+}^{\gm-2}u_{+}&0&-A_{1}\gm\r_{+}^{\gm-1}\\  
0& -A_{2}\a n_{+}^{\a-2}u_{+}&-A_{2}\a n_{+}^{\a-1}\\
-A_{1}\gm\r_{+}^{\gm-1}&-A_{2}\a n_{+}^{\a-1}& -(\r_{+}+n_{+})u_{+} 
\end{pmatrix}.
\end{split}
\end{equation}
Owing to $M_{+}=1$, it is easy to check that  three eigenvalues of the  matrix $\boldsymbol{M_{4}}$ satisfy: $\hat{\l}_{1}>0$, $\hat{\l}_{2}>0$, $\hat{\l}_{3}=0$. Take the coordinate transformation
\begin{equation}
\begin{split}
\begin{pmatrix}
\h\\ \b{\h}\\ \b{\s} 
\end{pmatrix}
=\boldsymbol{P} \begin{pmatrix}
\hat{\r}\\ \hat{n}\\ \hat{v} 
\end{pmatrix},
\label{bar rho u}
\end{split}
\end{equation}
where the  matrix  $\boldsymbol{P}$ is denoted by
 \begin{equation}
 \begin{aligned}
 \boldsymbol{P} =
 \begin{pmatrix}
 r_{11}& r_{21}&-\f{\r_{+}}{u_{+}} \\ 
 r_{12}& r_{22}&-\f{n_{+}}{u_{+}}\\ 
 r_{13}&r_{23}&1
 \end{pmatrix} {\rm with~constants~} r_{ij} ~{\rm for}~1\leq i\leq 2, ~1\leq j\leq 3, 
 \label{P}
 \end{aligned}
 \end{equation}
 such that 
 \begin{equation}
 \begin{split}
 (\h,\b{\h},\b{\s})~\boldsymbol{M_{4}}~(\h,\b{\h},\b{\s})^{\rm T}
 =(\hat{\r},\hat{n},\hat{v}) \begin{pmatrix} \hat{\l}_{1}&0&0\\ 
 0&\hat{\l}_{2}&0\\
 0&0&0
 \end{pmatrix}  (\hat{\r},\hat{n},\hat{v})^{\rm T}
 =\hat{\l}_{1}\hat{\r}^{2}+\hat{\l}_{2}\hat{n}^{2},
 \end{split}
 \label{P 1}
 \end{equation}
 By   
$(\ref{sigma 1})$, 
$(\ref{infty e2})$,  $(\ref{bound})$, and $(\ref{bar rho u})$-$(\ref{P 1})$, the third term is estimated as
\begin{equation}
\begin{aligned}
&
-a\nu \int \sg^{-(\nu-1)}G_{1}dx
\\
\geq&
 a \nu\int \sg^{-(\nu-1)} (\f{\hat{\l}_{1}}{2}\hat{\r}^{2}+\f{\hat{\l}_{2}}{2}\hat{n}^{2})dx+a \nu\int \sg^{-(\nu-2)}\f{A_{1}\gm(\gm+1)\r_{+}^{\gm}+A_{2}\a(\a+1)n_{+}^{\a}}{2|u_{+}|^{2}}\hat{v}^{2}dx
 \\
 & 
+a \nu\int \sg^{-(\nu-1)}\f{\r_{+}(u_{+}^{2}-A_{1}\gm\r_{+}^{\gm-1})}{|u_{+}|}\hat{v}(\s-\b{\s})dx-C\delta^{\f{1}{2}}\|\sg^{-\f{\nu-1}{2}} (\hat{\r},\hat{n}) \|^{2}-C\delta^{\f{1}{2}}\|\sg^{-\f{\nu-2}{2}} \hat{v} \|^{2}
\\
&
-C\delta^{\f{1}{2}}\| \sg^{-\f{\nu}{2}}(\s-\b{\s}) \|^{2} -C(\v+\delta) \| \sg^{-\f{\nu-2}{2}}(\h,\s,\b{\h},\b{\s}) \|^{2}
-C\v \| \sg^{-\f{\nu}{2}}(\h_{x},\s_{x},\b{\h}_{x},\b{\s}_{x}) \|^{2}
\\
&
-C\v\f{1}{\delta^{\nu}}(\h^{2}(t,0)+\b{\h}^{2}(t,0)),
\end{aligned}
\label{G1-3}
\end{equation}
where we have used the following facts
\be 
\begin{aligned} 
&
-(A_{1}\gm\wt{\r}^{\gm-1}-A_{1}\gm\r_{+}^{\gm-1})\geq \f{A_{1}\gm(\gm-1)\r_{+}^{\gm-1}}{|u_{+}|}\sg-C\sg^{2},
\\
&
-(A_{2}\a \wt{n}^{\a-1}-A_{2}\a n_{+}^{\a-1})\geq \f{A_{2}\a(\a-1)n_{+}^{\a-1}}{|u_{+}|}\sg-C\sg^{2},
\\
&
-(A_{1}\gm \wt{\r}^{\gm-2}\wt{u}-A_{1}\gm\r_{+}^{\gm-2}u_{+})\geq A_{1}\gm(3-\gm)\r_{+}^{\gm-2}\sg-C\sg^{2},
\\
&
-(A_{2}\a\wt{n}^{\a-2}\wt{v}-A_{2}\a n_{+}^{\a-2}u_{+})\geq A_{2}\a(3-\a)n_{+}^{\a-2}\sg-C\sg^{2}.
\end{aligned}
\nonumber
\ee 
 With the help of  $(\ref{sigma})$-$(\ref{sig})$, $(\ref{boundary d1})$, $(\ref{M_{+}=1 prior a})$,  $(\ref{infty e2})$-$(\ref{nonlinear f})$, $(\ref{bar rho u})$-$(\ref{P 1})$,
 and  $\s=\b{\s}+(\s-\b{\s})$, 
 it holds  that
\begin{equation}
\begin{aligned}
&
 -a\nu \int \sg^{-(\nu-1)}G_{2}(t,x)dx
 \\
\geq&
 -a^{2}\f{\mu+n_{+}}{2}\nu(\nu-1)  \|\sg^{-\f{\nu-2}{2}}\b{\s} \|^{2}-C(\delta+\v) \| \sg^{-\f{\nu-2}{2}}\b{\s} \|^{2}-C\v\|\sg^{-\f{\nu}{2}} \b{\s}_{x} \|^{2}-C\delta\|\sg^{-\f{\nu}{2}} (\b{\s}-\s) \|^{2}
 \\
\geq&
 -a^{2}\f{\mu+n_{+}}{2}\nu(\nu-1)  \|\sg^{-\f{\nu-2}{2}}\hat{v}\|^{2}-C\delta^{\f{1}{2}}(\| \sg^{-\f{\nu-1}{2}} (\hat{\r},\hat{n}) \|^{2}+\|\sg^{-\f{\nu-2}{2}}\hat{v} \|^{2})-C(\delta+\v) \| \sg^{-\f{\nu-2}{2}}\b{\s} \|^{2}
 \\
 &
 -C\v\| \sg^{-\f{\nu}{2}}\b{\s}_{x} \|^{2}
-C\delta\|\sg^{-\f{\nu}{2}} (\b{\s}-\s) \|^{2},
\label{forth term}
\end{aligned}
\end{equation}
\be 
\int \sg^{-\nu}n(\b{\s}-\s)^{2}dx\geq [n_{+}-C(\delta+\v)]\|\sg^{-\f{\nu}{2}}( \b{\s}-\s) \|^{2},~~~~~~~~~~~~~~~~~~~~~~~~~~~~~~~~~~~~~~~~~~~~~~~~~~~
\ee 
\begin{equation} 
\begin{aligned} 
&
 \int 	\sg^{-\nu} R_{1} dx
 \\
\geq&
 a\int\sg^{-(\nu-2)}[\f{A_{1}\gm (\gm-1)\r_{+}^{\gm-2}}{2}\h^{2}+\r_{+}\s^{2}+\f{A_{2}\a(\a-1)n_{+}^{\a-2}}{2}\b{\h}^{2}+n_{+}\b{\s}^{2}]dx
 \\
 &
 -C(\delta+\v)\|\sg^{-\f{\nu-2}{2}}(\h,\s,\b{\h},\b{\s})  \|^{2}\quad \quad
\\
\geq&
 a \f{A_{1}\gm(\gm+1)\r_{+}^{\gm}+A_{2}
\a(\a+1)n_{+}^{\a}}{2|u_{+}|^{2}}\| \sg^{-\f{\nu-2}{2}}\hat{v} \|^{2}-C\delta(\|\sg^{-\f{\nu-1}{2}} (\hat{\r},\hat{n}) \|^{2}+\| \sg^{-\f{\nu-2}{2}}\hat{v} \|^{2})
\\
& 
 -C(\delta+\v)\|\sg^{-\f{\nu-2}{2}}(\h,\s,\b{\h},\b{\s})  \|^{2}
.
\label{fifth term}
\end{aligned}
\end{equation}
For $\nu \in (0, 3]$,  with the help of $(\ref{sigma 1})$ and $(\ref{a})$,
 we add $(\ref{G1-3})$-$(\ref{fifth term})$  together   to have
\begin{equation}
\begin{aligned}
&
-a\nu \int \sg^{-(\nu-1)}G_{1}	dx-a\nu \int \sg^{-(\nu-1)}G_{2}dx+\int\sg^{-\nu}
n (\b{\s}-\s)^{2}dx+\int \sg^{-\nu} R_{1}dx
 \\
 \geq&
 c\| \sg^{-\f{\nu-1}{2}} (\hat{\r},\hat{n}) \|^{2}+\f{1}{4}\{a\f{A_{1}\gm(\gm+1)\r_{+}^{\gm}+A_{2}\a(\a+1)n_{+}^{\a}}{2|u_{+}|^{2}}[1+\nu-\f{\nu(\nu-1)}{2(1+b^{2})}]\|\sg^{-\f{\nu-2}{2}} \hat{v} \|^{2}
 \\
 &
+n_{+}\| \sg^{-\f{\nu}{2}}(\s-\b{\s}) \|^{2}\}+\int\sg^{-\nu}(\s-\b{\s},\hat{v})\boldsymbol{M_{5}}(\s-\b{\s},\hat{v})^{\rm T}dx-C\delta^{\f{1}{2}}\|\sg^{-\f{\nu-1}{2}} (\hat{\r},\hat{n})  \|^{2}
\\
&
 -C\delta^{\f{1}{2}}\| \sg^{-\f{\nu-2}{2}}\hat{v} \|^{2}-C(\v+\delta^{\f{1}{2}})\|  \sg^{-\f{\nu}{2}}(\s-\b{\s})\|^{2}-C(\v+\delta)\|\sg^{-\f{\nu}{2}}(\h_{x},\s_{x},\b{\h}_{x},\b{\s}_{x}) \|^{2}
 \\
 &
 -C\f{\v}{\delta^{\nu}}(\h^{2}(t,0)+\b{\h}^{2}(t,0)) 
 -C(\v+\delta)\| \sg^{-\f{\nu-2}{2}}(\h,\s,\b{\h},\b{\s}) \|^{2}
\nonumber
\end{aligned}\end{equation}
\begin{equation}
\begin{split}
&\geq 
c \|\sg^{-\f{\nu-2}{2}} (\h,\s,\b{\h},\b{\s}) \|^{2}+c\|\sg^{-\f{\nu-1}{2}}(\hat{\r},\hat{n})  \|^{2}+c\| \sg^{-\f{\nu}{2}}(\b{\s}-\s) \|^{2}~~~~~~~~~~~~~~~~~~~~~~~~~\\&
\quad -C(\delta+\v)\| \sg^{-\f{\nu}{2}}(\h_{x},\s_{x},\b{\h}_{x},\b{\s}_{x}) \|^{2}-C\v\f{\h^{2}(t,0)+\b{\h}^{2}(t,0)}{\delta^{\nu}}
 ,
	\label{lower bdd}
\end{split}
\end{equation}
where  
the positive  definite matrix  $\boldsymbol{M_{5}}$ is  defined by
\be
\boldsymbol{ M_{5}}=
\begin{pmatrix}
\f{3}{4}n_{+}&\f{  \sqrt{(\mu+n_{+})n_{+}}  }{2}ab \nu\sg\\
\f{\sqrt{(\mu+n_{+})n_{+}}}{2} ab\nu \sg   & \quad \f{3}{4}a \f{A_{1}\gm(\gm+1)\r_{+}^{\gm}+A_{2}\a(\a+1)n_{+}^{\a}}{2|u_{+}|^{2}}[(1+\nu)-\f{\nu(\nu-1)}{2(1+b^{2})}]\sg^{2}
\end{pmatrix}.
\ee 
 Then, we  consider the case $\nu \in [3, 2+\sqrt{8+\f{1}{1+b^{2}}})$ 
using the Lemma \ref{hardy lem} with $\zeta=\sg^{-(\nu-1)}$. Therefore,  with the aid of  $(\ref{a})$, the sixth term is estimated as below:
\begin{equation}
\begin{aligned}
& \int \sg^{-\nu}( \mu \s^{2}_{x}+n\b{\s}_{x}^{2})dx
\\
\geq&
a^{2}(\mu+n_{+})\f{(\nu-1)^{2}}{4}\|\sg^{-\f{\nu-2}{2}} \hat{v} \|^{2}-C\delta^{\f{1}{2}}(\|\sg^{-\f{\nu-2}{2}}\hat{v} \|^{2}+\| \sg^{-\f{\nu-1}{2}}(\hat{\r},\hat{n}) \|^{2})
\\
&
 -C(\delta+\v)\|\sg^{-\f{\nu}{2}} (\b{\s}-\s,\b{\s}_{x} )\|^{2}.
	\label{six term}
\end{aligned}
\end{equation}
For $\nu\in (3,\lambda]$, adding  $(\ref{lower bdd})$ to $(\ref{six term})$, taking $k=\nu B[4(1+b^{2})\nu+4b^{2}+5-\nu^{2}]^{-\f{1}{2}}\in (0,1)$, and  using  $(\ref{a})$ and $c|(\h,\b{\h},\b{\s})|\leq |(\hat{\r},\hat{n},\hat{v})| \leq C|(\h,\b{\h},\b{\s})|$, we  have 
\begin{equation}
\begin{aligned}
&
 -a\nu \int\sg^{-(\nu-1)}G_{1}dx-a\nu \int \sg^{-(\nu-1)}G_{2}dx+\int \sg^{-\nu} n(\b{\s}-\s)^{2}dx
 +\int\sg^{-\nu}R_{1}dx
 \\
 &
+\int\sg^{-\nu}(\mu \s_{x}^{2}+n\b{\s}_{x}^{2})dx
\\
 \geq&
  c\| \sg^{-\f{\nu-1}{2}} (\hat{\r},\hat{n}) \|^{2}+\int\sg^{-\nu}(\s-\b{\s},\hat{v})\boldsymbol{M_{6}}(\s-\b{\s},\hat{v})^{\rm T}dx +(1-k)\{n_{+}\| \sg^{-\f{\nu}{2}}(\s-\b{\s}) \|^{2}
  \\
 &+a\f{A_{1}\gm(\gm+1)\r_{+}^{\gm}+A_{2}\a(\a+1)n_{+}^{\a}}{2|u_{+}|^{2}}[1+\nu-\f{\nu(\nu-1)}{2(1+b^{2})}+\f{(\nu-1)^{2}}{4(1+b^{2})}]   
 \|\sg^{-\f{\nu-2}{2}} \hat{v} \|^{2}\}
 \\
 &
-C\delta^{\f{1}{2}}(\|\sg^{-\f{\nu-1}{2}} (\hat{\r},\hat{n})  \|^{2}+\| \sg^{-\f{\nu-2}{2}}\hat{v} \|^{2})-C(\v+\delta)\|\sg^{-\f{\nu}{2}}(\h_{x},\s_{x},\b{\h}_{x},\b{\s}_{x}) \|^{2}
\\
&
-C(\v+\delta^{\f{1}{2}})\|  \sg^{-\f{\nu}{2}}(\s-\b{\s})\|^{2}-C\f{\v}{\delta^{\nu}}(\h^{2}(t,0)+\b{\h}^{2}(t,0)) 
-C(\v+\delta)\| \sg^{-\f{\nu-2}{2}}(\h,\s,\b{\h},\b{\s}) \|^{2}
\\
\geq&
 c\| \sg^{-\f{\nu-2}{2}}(\h,\s,\b{\h},\b{\s}) \|^{2}+c\|\sg^{-\f{\nu-1}{2}}(\hat{\r},\hat{n})  \|^{2}+c\| \sg^{-\f{\nu}{2}}(\b{\s}-\s,\s_{x},\b{\s}_{x}) \|^{2}
 \\
 &
-C(\delta+\v)\| \sg^{-\f{\nu}{2}}(\h_{x},\b{\h}_{x}) \|^{2} -C\f{\v}{\delta^{\nu}} (\h^{2}(t,0)+\b{\h}^{2}(t,0))
,
\end{aligned}
\end{equation}
where  $\delta^{\f{1}{2}}$ and $\varepsilon$ are  small enough, and the positive definite matrix  $\boldsymbol{M_{6}}$ is defined as 
\be \boldsymbol{M_{6}}=
\begin{pmatrix}
kn_{+}&\f{ \sqrt{(\mu+n_{+})n_{+}}}{2} a b\nu \sg\\
\f{ \sqrt{(\mu+n_{+})n_{+}}}{2}ab\nu\sg   & \quad ka \f{A_{1}\gm(\gm+1)\r_{+}^{\gm}+A_{2}\a(\a+1)n_{+}^{\a}}{2|u_{+}|^{2}}[1+\nu-\f{\nu(\nu-1)}{2(1+b^{2})}+\f{(\nu-1)^{2}}{4(1+b^{2})}]\sg^{2}
	\end{pmatrix}.
\ee 
By $(\ref{sigma})$-$(\ref{sigma 1})$, $(\ref{bar rho u})$-$(\ref{P 1})$, Cauchy-Schwarz inequality  and $M_{+}=1$,
we estimate  terms on the  right hand side  as
\be 
\begin{aligned}
&	|\int \sg^{-\nu}R_{2}dx+\int \sg^{-\nu} R_{3}dx|
 \\
\leq&
 C\delta^{\f{1}{2}}\|\sg^{-\f{\nu-1}{2}}(\hat{\r},\hat{n})\|^{2}+C\delta^{\f{1}{2}}\| \sg^{-\f{\nu-2}{2}}(\h,\s,\b{\h},\b{\s}) \|^{2}+C\delta\|\sg^{-\f{\nu}{2}} (\b{\s}_{x},\b{\s}-\s) \|^{2}.
\label{R e}
\end{aligned}
\ee 
Finally, taking $ \delta^{\f{1}{2}}$ and $\varepsilon$   small enough, and combining $(\ref{second term e2})$-$(\ref{R e})$
,  we  obtain
\begin{align}
&
\frac{d}{dt}  	\int 	\sg^{-\nu} ( \mc{E}_{1}+\mc{E}_{2} )dx
+c  \|\sg^{-\f{\nu-2}{2}}(\h,\s,\b{\h},\b{\s}) \|^{2} 
+c \| \sg^{-\f{\nu}{2}}( \s_{x},\b{\s}_{x},\b{\s}-\s) \|^{2}
+\f{c}{\delta^{\nu}} [\h^{2}(t,0)+\b{\h}^{2}(t,0)]
\notag
\\
\leq&
  C(\delta+\v) \|\sg^{-\f{\nu}{2}} (\h_{x},\b{\h}_{x})\|^{2}. 
\label{r1}
\end{align}
%
%
%
Multiplying $(  \ref{r1}) $ by 
$(1+\delta  \tau)^{\xi}$
and integrating the resulted equation in $\tau$
over $[0,t]$,
we obtain $(\ref{L^{2} e2})$. The proof of Lemma \ref{lem L^{2} e2} is completed.
\end{proof}
In order to show  Proposition $\ref{prop M_{+}=1 time decay }$,  we need to obtain the high order weighted estimates of 
$  (\h,\s,\b{\h},\b{\s} )$.
\begin{lem}
	\label{lem high}
	Under the same conditions in Proposition $\ref{prop M_{+}=1 time decay }$, then the solution $(\h,\s,\b{\h},\b{\s})$ to the IBVP $(\ref{f1})$-$(\ref{boundary d1})$ satisfies  for $t \in[0,T]$ that
\begin{equation}
\begin{aligned}
&	
(1+\delta t)^{\xi} \| \sg^{-\f{\nu}{2}}(\h_{x},\b{\h}_{x}) \|^{2}
+\int_{0}^{t}(1+\delta \tau)^{\xi}\|\sg^{-\f{\nu}{2}} (\h_{x},\b{\h}_{x}) \|^{2}d\tau
\\
\leq&
 C (\| \sg^{-\f{\l}{2}}(\h_{0},\s_{0},\b{\h}_{0},\b{\s}_{0}) \|^{2}+\| \sg^{-\f{\l}{2}}(\h_{0x},\b{\h}_{0x})\|^{2})+C(\v+\delta)\int(1+\delta \tau)^{\xi}\|\sg^{-\f{\nu}{2}}(\s_{xx},\b{\s}_{xx}) \|^{2}d\tau
 \\
 &
+\delta\xi\int_{0}^{t}(1+\delta \tau)^{\xi-1}\| \sg^{-\f{\nu}{2}}(\h,\s,\h_{x},\b{\h},\b{\s},\b{\h}_{x}) \|^{2}d\tau,
\label{h-e-t}
\end{aligned}
\end{equation}
with  $\xi\geq 0$. 
\end{lem}
\begin{proof}
Adding $(\ref{h_{x}})$-$(\ref{bs_{x}})$ together, and multiplying the resulted equation by $\sg^{-\nu}$ with the weight function $\sg\geq0$ satisfying $(\ref{sig})$ and $(\ref{sg0})$,  we  integrate the resulted equation  in $x$ over $\mb{R}_{+}$ to  obtain
	\be 
\begin{aligned} 
	&
	\f{d}{dt}\int \sg^{-\nu}(\mu \f{\h_{x}^{2}}{2}+\f{\b{\h}_{x}^{2}}{2}+\wt{\r}^{2}\h_{x}\s +\wt{n}\b{\h}_{x}\b{\s})dx-[\sg^{-\nu}(\mu u\f{\h_{x}^{2}}{2}+v\f{\b{\h}_{x}^{2}}{2}-\wt{\r}^{2}\h_{t}\s-\wt{n}\b{\h}_{t}\b{\s})](t,0)
	\\
	&
-a\nu \int\sg^{-(\nu-1)}[\mu u\f{\h_{x}^{2}}{2}+v\f{\b{\h}_{x}^{2}}{2}-\wt{\r}^{2}\h_{t}\s-\wt{n}\b{\h}_{t}\b{\s}] dx	
	+\int \sg^{-\nu}(\wt{\r}^{2}\f{p'_{1}(\r)}{\r}\h_{x}^{2}
	+\wt{n}\f{p'_{2}(n)}{n}\b{\h}_{x}^{2})dx
	\\
	=&
	\sum_{i=1}^{6}\mc{J}_{i},
	\label{1f-1}
\end{aligned}
\ee 
where
\be 
\begin{aligned}
	\mc{J}_{1}=&
	-\int\sg^{-\nu}[ \wt{\r}^{2}(\h_{t}+u\h_{x})\s_{x}+\wt{n}(\b{\h}_{t}+v\b{\h}_{x})\b{\s}_{x}+2\wt{\r}\wt{\r}_{x}\h_{t}\s+\wt{n}_{x}\b{\h}_{t}\b{\s}]dx,
	\\
	\mc{J}_{2}=&
	\int\sg^{-\nu} [-\mu \h\h_{x}\s_{xx}-\b{\h}\b{\h}_{x}\b{\s}_{xx}+ \mu\wt{\r}^{2}(\f{1}{\r}-\f{1}{\wt{\r}})\h_{x}\s_{xx}+\wt{n}(\f{1}{n}-\f{1}{\wt{n}})(\wt{n}\b{\s}_{x})_{x}\b{\h}_{x}]dx, ~~~~~~~~~~~~~~
	\\
	\mc{J}_{3}=&
	\int\sg^{-\nu}[\wt{\r}^{2}\f{n}{\r}(\b{\s}-\s)\h_{x}-\wt{n}(\b{\s}-\s)\b{\h}_{x}]dx,
\\
	\mc{J}_{4}=&
	-\int\sg^{-\nu} ( \f{3}{2}\mu \s_{x}\f{\h_{x}^{2}}{2}+\f{3}{2}\b{\s}_{x}\b{\h}_{x}^{2})dx+\int\sg^{-\nu} \wt{n}\f{(\b{\h}\b{\s}_{x})_{x}}{n}\b{\h}_{x}dx,\quad 
	\mc{J}_{5}=\int\sg^{-\nu}( F_{1}\wt{\r}^{2}\h_{x}+F_{2}\wt{n}\b{\h}_{x})dx,
	\\	
	\mc{J}_{6}=&
	-\int\sg^{-\nu} [\mu (
	\f{1}{2}\h_{x}\wt{u}_{x}+\s_{x}\wt{\r}_{x})\h_{x}+\mu (\h\wt{u}_{x}+\s\wt{\r}_{x})_{x}\h_{x}
	\\
	&
	-\wt{n}_{x}\b{\h}_{x}\b{\s}_{xx}+(
	\f{1}{2}\b{\h}_{x}\wt{v}_{x}+\b{\s}_{x}\wt{n}_{x})\b{\h}_{x}-(\b{\h}\wt{v}_{x}+\b{\s}\wt{n}_{x})_{x}\b{\h}_{x}]dx.
	\nonumber
\end{aligned}
\ee 
Owing to $(\ref{sigma})$, $(\ref{boundary d1})$, and $(\ref{infty e2})$, we  obtain the estimates for terms on the left hand side as below
 \begin{align}
 -&[\sg^{-\nu}(\mu u\f{\h_{x}^{2}}{2}+v\f{\b{\h}_{x}^{2}}{2}-\wt{\r}\h_{x}\s-\wt{n}\b{\h}_{x}\b{\s})](t,0)
 \geq \f{c}{\delta^{\nu}}(\h^{2}_{x}(t,0)+\b{\h}^{2}_{x}(t,0))
 \geq 0,~~~~~~~~~~
 \label{h-l1}
 \\
& 
-a\nu \int\sg^{-(\nu-1)}(\mu u\f{\h_{x}^{2}}{2}+v\f{\b{\h}_{x}^{2}}{2})dx
\notag
\\ \geq& \f{a\nu |u_{+}| }{2}\int\sg^{-(\nu-1)}(\mu \h_{x}^{2}+\b{\h}_{x}^{2})dx
-C(\v+\delta)\|\sg^{-\f{\nu-1}{2}} (\h_{x},\b{\h}_{x})\|^{2},
\\
& 
 -a\nu \int \sg^{-(\nu-1)}(-\wt{\r}^{2}\h_{t}\s-\wt{n}\b{\h}_{t}\b{\s})dx
 \notag
 \\
\geq &
-C\v\|\sg^{-\f{\nu}{2}} (\h_{x},\b{\h}_{x}) \|^{2}-C\| \sg^{-\f{\nu-2}{2}}(\h,\s,\b{\h},\b{\s}) \|^{2}-C\|\sg^{-\f{\nu}{2}}(\s_{x},\b{\s}_{x})  \|^{2},
\end{align}
\begin{align}
 \int& \sg^{-\nu}(\wt{\r}^{2}\f{p'_{1}(\r)}{\r}\h_{x}^{2}
+\wt{n}\f{p'_{2}(n)}{n}\b{\h}_{x}^{2})dx
\geq \f{A_{1}\gm \r_{+}^{\gm}}{2}\| \sg^{-\f{\nu}{2}}\h_{x} \|^{2}+\f{A_{2}\a n_{+}^{\a-1}}{2}\|\sg^{-\f{\nu}{2}}\b{\h}_{x} \|^{2},~
\end{align}		
where we take  $\delta $ and $\v$ small enough. \\
We turn to estimate terms on the right hand side of $(\ref{1f-1})$.
With the help of  $(\ref{sigma})$, $(\ref{infty e2})$,  Young inequality and  Cauchy-Schwarz inequality, we  gain
\begin{align} 
|\mc{J}_{1}|\leq &
C\delta \| \sg^{-\f{\nu}{2}}(\h_{x},\b{\h}_{x}) \|^{2}+ C\| \sg^{-\f{\nu}{2}}(\s_{x},\b{\s}_{x}) \|^{2} +C\| \sg^{-\f{\nu-2}{2}}(\h,\s,\b{\h},\b{\s}) \|^{2},
\\
|\mc{J}_{2}|\leq &
C(\| (\h,\b{\h})\|_{L^{\infty}}+\delta)\| \sg^{-\f{\nu}{2}}(\h_{x},\b{\h}_{x},\s_{xx},\b{\s}_{xx})\|^{2}+C\delta^{2}\|\sg^{-\f{\nu}{2}}\b{\s}_{x} \|^{2}
\notag
\\
\leq&
 C(\delta+\v)\| \sg^{-\f{\nu}{2}}(\h_{x},\b{\h}_{x}) \|^{2}+C(\delta+\v)\|\sg^{-\f{\nu}{2}} (\s_{xx},\b{\s}_{xx}) \|^{2}+C\delta \|\sg^{-\f{\nu}{2}} \b{\s}_{x} \|^{2},
\\
|\mc{J}_{3}|\leq&
 \f{A_{1}\gm\r_{+}^{\gm}}{8}\| \sg^{-\f{\nu}{2}}\h_{x} \|^{2}+\f{A_{2}\a n_{+}^{\a-1}}{8}\|\sg^{-\f{\nu}{2}}\b{\h}_{x}  \|^{2}+C\|\sg^{-\f{\nu}{2}}( \b{\s}-\s )\|^{2},
 \\
|\mc{J}_{4}| 
\leq&
 C\v\| \sg^{-\f{\nu}{2}}(\h_{x},\b{\h}_{x}) \|^{2}+ C\v\| \sg^{-\f{\nu}{2}}(\s_{xx},\b{\s}_{xx}) \|^{2}+C\v\|\sg^{-\f{\nu}{2}} (\s_{x},\b{\s}_{x}) \|^{2},
\\
|\mc{J}_{5}| \leq &
 C\delta \|\sg^{-\f{\nu}{2}}(\h_{x},\b{\h}_{x}) \|^{2}+C\delta \| \sg^{-\f{\nu-2}{2}}(\h,\s,\b{\h},\b{\s}) \|^{2},~~~~~~~~~~~~~~~~~~~~~~~
\\|\mc{J}_{6}| \leq&
 C\delta\| \sg^{-\f{\nu}{2}}(\h_{x},\b{\h}_{x}) \|^{2}+C\delta^{2}\| \sg^{-\f{\nu}{2}}\b{\s}_{xx} \|^{2}+C\delta \|\sg^{-\f{\nu}{2}} (\s_{x},\b{\s}_{x}) \|^{2}+C\delta \| \sg^{-\f{\nu-2}{2}}(\h,\s,\b{\h},\b{\s}) \|^{2}.
\label{h-r6}
\end{align}
Finally,  the substitution of $(\ref{h-l1})$-$( \ref{h-r6})$ into $(\ref{1f-1})$ for $\delta $ and $\varepsilon$  small enough leads to that 
\begin{equation}
\begin{aligned}
&
 \frac{d}{dt}	\int \sg^{-\nu}[  \mu \f{\h_{x}^{2}}{2}+\f{\b{\h}_{x}^{2}}{2} 	+\wt{\r}^{2}\h_{x}\s+\wt{n} \b{\h}_{x}\b{\s} ]	dx     +c \|\sg^{-\f{\nu-1}{2}}  (\h_{x},\b{\h}_{x} ) \|^{2}  
 \\
 &
+\f{A_{1}\gm\r_{+}^{\gm}}{4}\|\sg^{-\f{\nu}{2}} \h_{x} \|^{2}+\f{A_{2}\a n_{+}^{\a-1}}{4} \|\sg^{-\f{\nu}{2}}\b{\h}_{x} \|^{2}
\\
\leq& 
C(\v+\delta)\| \sg^{-\f{\nu}{2}}(\s_{xx},\b{\s}_{xx}) \|^{2}+C\| \sg^{-\f{\nu-2}{2}}(\h,\s,\b{\h},\b{\s}) \|^{2}+C\| \sg^{-\f{\nu}{2}} (\s_{x},\b{\s}_{x},\b{\s}-\s) \|^{2}.~~~~~~~~~
\label{high order weighted e}
\end{aligned}
\end{equation}
Multiplying $(\ref{high order weighted e})$  by $(1+\delta \tau)^{\xi}$ and integrating the resulted equation  in $\tau$ over $[0,t]$, and using Cauchy-Schwarz inequality and Lemma $\ref{lem L^{2} e2}$, we obtain the desired estimate $(\ref{h-e-t})$. The proof of Lemma $\ref{lem high}$ is completed.
%
\end{proof}

\begin{lem}
	\label{psi_{xx} estimate lem}	
Under the same conditions in Proposition $\ref{prop M_{+}=1 time decay }$, then the solution $(\h,\s,\b{\h},\b{\s})$ to the IBVP $(\ref{f1})$-$(\ref{boundary d1})$ satisfies  for $t \in[0,T]$ that
\begin{equation}
\begin{aligned}
&
(1+ \delta t)^{\xi} \|\sg^{-\f{\nu}{2}}(\s_{x},\b{\s}_{x})\|^{2}	+\int^{t}_{0} 	( 1+\delta \tau  )^{\xi}	\|\sg^{-\f{\nu}{2}}(\s_{xx},\b{\s}_{xx})\|^{2} d\tau 
\\
\leq&
C\| \sg^{-\f{\l}{2}}(\h_{0},\s_{0},\h_{0x},\s_{0x},\b{\h}_{0},\b{\s}_{0},\b{\h}_{0x},\b{\s}_{0x})  \|^{2} 
\\
&
+C \delta  \xi 	\int^{t}_{0}	( 1   +\delta \tau  )^{\xi-1 }	\|\sg^{-\f{\nu}{2}}  ( \h,\s,\h_{x},\s_{x},\b{\h},\b{\s},\b{\h}_{x},\b{\s}_{x} )  \|^{2} d\tau,
\label{psi_{xx} time e 1}
\end{aligned}
\end{equation}
	with $\xi\geq0$.
\end{lem}

\begin{proof}

	Multiplying $(\ref{f1})_{2}$ by $-\sg^{-\nu}\psi_{xx}$,  $(\ref{f1})_{4}$ by $-\sg^{-\nu}\b{\s}_{xx}$  respectively with the function $\sg$ satisfying $(\ref{sig})$ and $(\ref{sg0})$,  then adding them together and integrating the resulted equation in $x$ over $\mathbb{R}_{+}$  lead to
	\begin{equation}
	\begin{aligned}
	&
\f{d}{dt}\int \sg^{-\nu} (\f{\s_{x}^{2}}{2}+\f{\b{\s}_{x}^{2}}{2})dx-a\nu \int \sg^{-(\nu-1)}(\s_{t}\s_{x}+\b{\s}_{t}\b{\s}_{x})dx+\int\sg^{-\nu} (\f{\mu}{\r}\s_{xx}^{2}+\b{\s}_{xx}^{2})dx=\sum_{i=1}^{3}\mc{K}_{i},
	\label{psi_{xx} space w 1}
	\end{aligned}
	\end{equation}
	where 
%
\begin{equation}
\begin{aligned}
\mc{K}_{1}	=&
\int\sg^{-\nu} [ 	u \psi_{x} \psi_{xx} 	-\f{n}{\r}(\b{\s}-\s)\s_{xx} +\f{p'_{1}(\r)}{\r}\h_{x}\s_{xx} +v \b{\s}_{x}\b{\s}_{xx} +(\b{\s}-\s)\b{\s}_{xx}	+\f{p_{2}{'}(n)}{n}\b{\h}_{x}\b{\s}_{xx}]	dx,
\\
\mc{K}_{2} 	=&
\int\sg^{-\nu}  [  \widetilde{u}_{x} \psi \psi_{xx}  	-\mu \widetilde{u}_{xx} ( \frac{1}{\rho}-\f{1}{\wt{\r}})\s_{xx}+(\f{p'_{1}(\r)}{\r}-\f{p'_{1}(\wt{\r})}{\wt{\r}})\s_{xx}\wt{\r}_{x}-(\f{n}{\r}-\f{\wt{n}}{\wt{\r}})(\wt{v}-\wt{u})\s_{xx} 	
+\widetilde{v}_{x} \b{\s} \b{\s}_{xx}
\\
&  	
+ (\wt{n}\wt{v}_{x})_{x} ( \frac{1}{n}-\f{1}{\wt{n}})\b{\s}_{xx}+(\f{p'_{2}(n)}{n}-\f{p'_{2}(\wt{n})}{\wt{n}})\b{\s}_{xx}\wt{n}_{x}-\f{\wt{n}_{x}}{\wt{n}}\b{\s}_{x}\b{\s}_{xx} -\f{(\b{\h}\wt{v}_{x})_{x}}{n}\b{\s}_{xx}]	dx,
\\	
\mc{K}_{3} =&
-\int\sg^{-\nu}  [\f{(\b{\h}\b{\s}_{x})_{x}}{n}\b{\s}_{xx}+(\f{1}{n}-\f{1}{\wt{n}})(\wt{n}\b{\s}_{x})_{x}\b{\s}_{xx}]	dx.
\nonumber
\end{aligned}
\end{equation}
%
First, we estimate terms in the left side of  $(\ref{psi_{xx} space w 1})$.
The second term  is estimated as follows:
\begin{equation}
\begin{aligned}
&a\nu \int\sg^{-(\nu-1)}(\s_{t}\s_{x}+\b{\s}_{t}\b{\s}_{x}) dx
\\
\leq&
 C\delta \|\sg^{-\f{\nu}{2}} (\s_{xx},\b{\s}_{xx}) \|^{2}+C\delta \|\sg^{-\f{\nu}{2}} (\h_{x},\s_{x},\b{\h}_{x},\b{\s}_{x},\b{\s}-\s) \|^{2}+C\delta \| \sg^{-\f{\nu-2}{2}}(\h,\s,\b{\h},\b{\s}) \|^{2}
  ,
\label{psi_{xx} second term e}
\end{aligned}
\end{equation}
where we have used $(\ref{sigma})$, $(\ref{f1})_{2}$, $(\ref{f1})_{4}$, $(\ref{infty e2})$  and Cauchy-Schwarz inequality.\\
With the help of $\delta $  and $\v$ small enough, the third term  is estimated as follows:
\begin{equation}
\begin{aligned}
\int \sg^{-\nu} (\f{\mu}{\r}\s_{xx}^{2}+\b{\s}_{xx}^{2})dx\geq& [\f{\mu}{\r_{+}}-C(\v+\delta)]\|\sg^{-\f{\nu}{2}} \s_{xx} \|^{2}+\| \sg^{-\f{\nu}{2}}\b{\s}_{xx} \|^{2}
\\
\geq&
 \f{\mu}{2\r_{+}}\|\sg^{-\f{\nu}{2}}\s_{xx} \|^{2}+\| \sg^{-\f{\nu}{2}}\b{\s}_{xx} \|^{2}.
\end{aligned}
\end{equation}
%
 We turn to estimate  terms on the right hand side of $(\ref{psi_{xx} space w 1})$. With the help of  $(\ref{sigma})$, $(\ref{infty e2})$  and Cauchy-Schwarz inequality, we  obtain
\begin{align}
&
| \mc{K}_{1} |  	\leq  	\f{\mu}{8\r_{+}}\|\sg^{-\f{\nu}{2}} \s_{xx} \|^{2}+\f{1}{8}\| \sg^{-\f{\nu}{2}}\b{\s}_{xx} \|^{2}+C\| \sg^{-\f{\nu}{2}}(\h_{x},\s_{x},\b{\h}_{x},\b{\s}_{x},\b{\s}-\s) \|^{2},
\\
&
| \mc{K}_{2} | 	\leq C \delta\|\sg^{-\f{\nu}{2}} (\s_{xx},\b{\s}_{xx}) \|^{2}+C\delta\| \sg^{-\f{\nu-2}{2}}(\h,\s,\b{\h},\b{\s}) \|^{2}+C\delta\|\sg^{-\f{\nu}{2}} (\b{\h}_{x},\b{\s}_{x}) \|^{2}
,
\label{K_{1} e}
\\&
| \mc{K}_{3} | 
\leq C(\v+\delta) \|\sg^{-\f{\nu}{2}}(\b{\s}_{x},\b{\s}_{xx}) \|^{2}.
\label{K_{4} e}
\end{align}
	%
Finally, we substitute  $(\ref{psi_{xx} second term e})$-$(\ref{K_{4} e})$ into $(\ref{psi_{xx} space w 1})$ to gain under the condition $\delta$ and $\v$ small enough that
\begin{equation}
\begin{aligned}
&
\frac{d}{dt} 	\int 	\sg^{-\nu}(\f{\s_{x}^{2}}{2}+\f{\b{\s}_{x}^{2}}{2})dx  	+	\f{\mu}{4\r_{+}}\|\sg^{-\f{\nu}{2}}\s_{xx} \|^{2}+\f{1}{4} \|\sg^{-\f{\nu}{2}}\b{\s}_{xx} \|^{2}
\\
\leq&
C\|\sg^{-\f{\nu}{2}} (\h_{x},\s_{x},\b{\h}_{x},\b{\s}_{x},\b{\s}-\s) \|^{2}+C\delta\| \sg^{-\f{\nu-2}{2}}(\h,\s,\b{\h},\b{\s}) \|^{2}
\label{psi_{xx} space weighted e}
\end{aligned}
\end{equation}
Multiplying $(\ref{psi_{xx} space weighted e})$ by $ (1+\delta \tau)^{\xi } $ and integrating the resulted inequality in $\tau $ over $[0,t]$,  and using 
Lemmas $\ref{lem L^{2} e2}$-$\ref{lem high}$ 
 and the smallness of $\delta$ and $\varepsilon$, we  obtain $(\ref{psi_{xx} time e 1})$. The proof of Lemma $\ref{psi_{xx} estimate lem}$   is completed.
\end{proof}
\textit{\underline{Proof of Proposition \ref{prop M_{+}=1 time decay }}}  
With the help of Lemmas $\ref{lem L^{2} e2}$-$\ref{psi_{xx} estimate lem}$, it holds for $\delta $ and $\varepsilon $ suitably small that
\begin{align}
&
 (1+ \delta t)^{\xi} \|\sg^{-\f{\nu}{2}}(\h,\s,\b{\h},\b{\s})\|^{2}_{1}	+\int^{t}_{0} 	( 1+\delta \tau  )^{\xi}	\|\sg^{-\f{\nu-2}{2}}(\h,\s,\b{\h},\b{\s})\|^{2} d\tau 
 +\int_{0}^{t}(1+\delta \tau)^{\xi} \| \sg^{-\f{\nu}{2}}(\h_{x},\s_{x},\b{\h}_{x},\b{\s}_{x}) \|^{2}
 \notag
  \\
 &
 +C\int_{0}^{t}(1+\delta \tau)^{\xi} \|\sg^{-\f{\nu}{2}} (\s_{xx},\b{\s}_{xx},\b{\s}-\s) \|^{2}
+C\f{1}{\delta^{\nu}}\int_{0}^{t}(1+\delta \tau )^{\xi}(\h^{2}(t,0)+\b{\h}^{2}(t,0))d\tau
\notag
\\
\leq&
 C\| \sg^{-\f{\l}{2}}(\h_{0},\s_{0},\b{\h}_{0},\b{\s}_{0})  \|^{2}_{1}+C \delta  \xi 	\int^{t}_{0}	( 1   +\delta \tau  )^{\xi-1 }	\| \sg^{-\f{\nu}{2}} ( \h,\s,\b{\h},\b{\s} )  \|^{2}_{1}  d\tau,
\label{high order space time e} 
\end{align}
where  $C>0$ is a positive constant  independent of $T$, $\nu$ and $\xi$.
Applying  similar induction arguments as in  \cite{KM1,NM,CHS}  to $(\ref{high order space time e})$, we have
\begin{equation}
\begin{aligned}
&
 ( 1+ \delta t )^{\f{\l-\nu}{2} + \beta}  \|\sg^{-\f{\nu}{2}} ( \h,\s,\h_{x},\s_{x},\b{\h},\b{\s},\b{\h}_{x},\b{\s}_{x}) \|^{2}  +  \int^{t}_{0}  ( 1  +\delta \tau )^{ \f{\l-\nu}{2}   +\beta}  \|\sg^{-\f{\nu-2}{2}} ( \h,\b{\h},\s,\b{\s} ) \|^{2}
 \\
 &
+\int^{t}_{0} ( 1  +\delta \tau )^{ \f{\l-\nu}{2} +\beta } \| \sg^{-\f{\nu}{2}}(\h_{x},\s_{x},\s_{xx},\b{\h}_{x},\b{\s}_{x},\b{\s}_{xx},\b{\s}-\s)  \|^{2}
\\
\leq&
 C( 1  +\delta t )^{ \beta }  \|\sg^{-\f{\l}{2}} (\h_{0},\s_{0},\h_{0x},\s_{0x},\b{\h}_{0},\b{\s}_{0},\b{\h}_{0x},\b{\s}_{0x})  \|^{2}
\end{aligned}
\end{equation}
for $\beta>0$,
which implies
\begin{equation}
\| \sg^{-\f{\nu}{2}}( \h,\s,\h_{x},\s_{x},\b{\h},\b{\s},\b{\h}_{x},\b{\s}_{x}) (t)\|^{2} \leq  C ( 1+\delta t )^{- \frac{\lambda-\nu  }{4}} \|\sg^{-\f{\l}{2}}(\h_{0},\s_{0},\h_{0x},\s_{0x},\b{\h}_{0},\b{\s}_{0},\b{\h}_{0x},\b{\s}_{0x})  \|^{2}.
\end{equation} 
\textbf{Acknowledgments}

The research of the paper is supported by the National Natural Science Foundation of China (Nos. 11931010, 11871047, 11671384), by the key research project of Academy for Multidisciplinary Studies, Capital Normal University, and by the Capacity Building for Sci-Tech Innovation-Fundamental Scientific Research Funds (No. 007/20530290068).


\end{document}